\documentclass[a4paper,11pt,reqno]{amsart}

\usepackage[utf8]{inputenc}
\RequirePackage[l2tabu, orthodox]{nag}
\usepackage[T1]{fontenc}
\usepackage{lmodern}
\usepackage{amsthm,amssymb,amsmath}
\usepackage{latexsym}
\usepackage{graphicx}
\usepackage{subfigure}
\usepackage{tikz}
\usetikzlibrary{arrows}
\usetikzlibrary{calc}
\usepackage[linktocpage=true]{hyperref}
\usepackage[left=3.8cm, right=3.8cm, top=3cm, bottom=3cm]{geometry}
\usepackage{color}
\usepackage{here}
\usepackage{mathrsfs}
\usepackage{pgf,tikz}
\usetikzlibrary{decorations.pathreplacing, shapes.multipart, arrows, matrix, shapes}
\usetikzlibrary{patterns}
\usepackage{caption}
\usepackage[frenchb,english]{babel}
\usepackage[all] {xy}
\usepackage{enumitem}

\makeatletter
\newcommand\footnoteref[1]{\protected@xdef\@thefnmark{\ref{#1}}\@footnotemark}
\makeatother

\hypersetup{ colorlinks   = true, 
urlcolor  = blue, 
linkcolor    = blue, 
citecolor   = blue 
}

\def\beq{\begin{equation}}
\def\eeq{\end{equation}}

\def\MLS{\mathrm{MLS}}
\def\MAS{\mathrm{MAS}}
\def\QC{\mathrm{QC}}
\def\MKD{\mathrm{MKD}}
\def\FKD{\mathrm{FKD}}

\def\PSL{\mathrm{PSL}_2(\mathbb{R})}

\def\PSLc{\mathrm{PSL}_2(\mathbb{C})}
\def\H2{\mathbb{H}^2}
\def\Hyp{\mathbb{H}^3}

\def\acts{\curvearrowright}

\def\restriction#1#2{\mathchoice
              {\setbox1\hbox{${\displaystyle #1}_{\scriptstyle #2}$}
              \end{enumerate}
\restrictionaux{#1}{#2}}
              {\setbox1\hbox{${\textstyle #1}_{\scriptstyle #2}$}
              \restrictionaux{#1}{#2}}
              {\setbox1\hbox{${\scriptstyle #1}_{\scriptscriptstyle #2}$}
              \restrictionaux{#1}{#2}}
              {\setbox1\hbox{${\scriptscriptstyle #1}_{\scriptscriptstyle #2}$}
              \restrictionaux{#1}{#2}}}
\def\restrictionaux#1#2{{#1\,\smash{\vrule height .8\ht1 depth .85\dp1}}_{\,#2}}

\newcommand{\R}{{\mathbb R}}

\newcommand{\C}{{\mathbb C}}
\newcommand{\D}{{\mathbb D}}

\newcommand{\PP}{{\mathbb P}}

\newcommand{\loc}{{\operatorname{loc}}}

\newcommand{\cC}{{\mathcal C}}

\newcommand{\cF}{{\mathcal F}}
\newcommand{\ctF}{{\widetilde{\mathcal F}}}

\newcommand{\ctG}{{\widetilde{\mathcal G}}}

\newcommand{\cL}{{\mathcal L}}

\newcommand{\cS}{{\mathcal S}}

\newcommand{\eps}{\varepsilon}

\newcommand{\length}{\operatorname{length}}

\def\dans{\mathop{\subset}}
\newcommand{\Leb}{\mathrm{Leb}}
\newcommand{\sect}{\mathrm{sect}}

\newcommand{\Vol}{\mathrm{Vol}}
\newcommand{\tM}{\widetilde{M}}
\newcommand{\bord}{\partial_{\infty}}
\newcommand{\moins}{\setminus}

\newcommand{\Ext}{\mathrm{Ext}}
\newcommand{\Id}{\mathrm{Id}}
\newcommand{\Isom}{\mathrm{Isom}}
\newcommand{\Area}{\mathrm{Area}}
\newcommand{\dist}{\mathrm{dist}}

\def\To{\mathop{\longrightarrow}}

\newtheorem{theorem}{Theorem}[section]
\newtheorem{remark}[theorem]{Remark}

\newtheorem{lemma}[theorem]{Lemma}

\newtheorem{defi}[theorem]{Definition}
\newtheorem{prop}[theorem]{Proposition}
\newtheorem{corollary}[theorem]{Corollary}

\newtheorem{question}[theorem]{Question}

\begin{document}
	
\title[]{Foliated Plateau problems, geometric rigidity and equidistribution of closed $k$-surfaces}

\author[]{Sébastien Alvarez}
\address{}
\email{}

\date{\today}

\begin{abstract} 
In this note, we survey recent advances in the study of dynamical properties of the space of surfaces with constant curvature in three-dimensional manifolds of negative sectional curvature. We interpret this space as a two-dimensional analogue of the geodesic flow and explore the extent to which the thermodynamic properties of the latter can be generalized to the surface setting. Additionally, we apply this theory to derive geometric rigidity results, including the rigidity of the hyperbolic marked area spectrum.
\end{abstract}

\maketitle

\section{Introduction}

\subsection{Foliated Plateau problems} 

The abundance of closed geodesics in a closed, negatively curved $3$-dimensional Riemannian manifold can be used to obtain many geometric rigidity results. For instance, knowing their exponential growth rate is sufficient to characterize the hyperbolic metric among all negatively curved metrics with a fixed volume. This is famously known as the Besson-Courtois-Gallot rigidity theorem \cite{BCG}. The dynamical study of the geodesic flow is fundamental because, as Gromov states in \cite{Gromov_FolPlateau1}: \emph{``if one wishes to understand closed geodesics not as individuals but as members of a community, one has to look at all (not only closed) geodesics.''} This forms the basis of Gromov’s theory of \emph{foliated Plateau problems}.

In this paper, we focus on the dynamical study of higher-dimensional analogues of the geodesic flow, specifically the distribution and counting of closed surfaces in closed, negatively curved $3$-manifolds. The primary question is: what geometric properties should our surfaces satisfy to generalize periodic geodesics effectively?

Totally geodesic surfaces are relatively rare. Examples of
closed hyperbolic 3-manifolds without closed totally geodesic surfaces are given in \cite[\S 5.3]{MR03} and most hyperbolic knot complements have no totally geodesic surface: \cite{Basilio_Lee_Malionek,Reid}. On the other hand, it has been recently proven that the presence of infinitely many closed totally geodesic surfaces in a closed hyperbolic $3$-manifold $(M,h_0)$  implies that the manifold is arithmetic (see \cite{BFMS,MM}), which is a topological restriction on the manifold. There are two natural candidates for representing conjugacy classes of surface subgroups.

\begin{itemize}
\item \emph{Minimal surfaces}, which are surfaces whose mean curvature $H$ (the half trace of the shape operator) equals zero.
\item \emph{$k$-surfaces}, which are convex surfaces whose extrinsic curvature $\kappa_{ext}$ (the determinant of the shape operator) equals a positive constant $k>0$.
\end{itemize}

These surfaces naturally arise as solutions to PDEs: the minimal surface PDE for minimal surfaces, and a Monge-Ampère equation for $k$-surfaces. Following Gromov’s perspective \cite{Gromov_FolPlateau1,Gromov_FolPlateau2}, one should consider compact solutions of these PDEs not as individuals but as part of a community of all PDE solutions. These solutions give a laminated structure in a phase space, analogous to the geodesic flow. The goal is to derive information from the dynamical and ergodic properties of this lamination.

\subsection{Labourie's lamination and phase space} 
In the early 2000s, Labourie initiated a dynamical study of $k$-surfaces, revealing its richness (see \cite{LabourieGAFA,LabourieInvent,LabourieAnnals}). He showed how to solve asymptotic Plateau problems for $k$-surfaces. Consider an example of such a problem relevant to our discussion:

\emph{Let $(X, h)$ be a $3$-dimensional Cartan-Hadamard Riemannian manifold with pinched negative curvature $-a^2\leq\sect\leq -1$ and a constant $k\in(0,1)$. Given an oriented Jordan curve $c\dans\bord X$, the problem is to prove the existence and uniqueness of an embedded, oriented disc $D$ inside $X$ such that $\kappa_{ext}=k$ over $D$ and whose ideal boundary $\bord D$ coincides with $c$.}

Existence and uniqueness in this context were first proven by Rosenberg and Spruck for the hyperbolic space $\Hyp$ \cite{Rosenberg_Spruck}, and then in the general case by Labourie \cite{LabourieInvent}. Labourie’s work is more complete as it considers more general boundary conditions and classifies all admissible ones (see also \cite{Smith_asymp}). The existence of the solution of the asymptotic Plateau problem for minimal surfaces in the hyperbolic case was proven by Anderson \cite{Anderson}, but uniqueness does not generally hold unless the boundary satisfies certain geometric conditions, such as being a $C$-quasicircle with $C\simeq 1$ \cite{Uhlenbeck}. Recently, Huang, Lowe, and Seppi provided an example of a quasicircle spanning uncountably many minimal discs \cite{Huang_Lowe_Seppi}.

Using an approach based on the study of holomorphic curves, Labourie showed in \cite{LabourieGAFA} that the infinite-dimensional space of all marked immersed $k$-surfaces in $X$ naturally compactifies with a finite-dimensional boundary. This \emph{phase space} has the structure of a compact Riemann surface lamination, analogous to the geodesic flow, and shares remarkable dynamical properties with it (see \cite{LabourieInvent,LabourieAnnals}). However, some questions remain open, particularly regarding the distribution of closed $k$-surfaces in compact quotients of $X$.

\subsection{Asymptotic counting of closed minimal surfaces} 

In a recent paper \cite{CMN}, Calegari, Marques, and Neves introduced a new idea in the context of minimal surfaces. This idea builds on Kahn and Marković's work, which demonstrates the abundance of quasi-Fuchsian subgroups within the fundamental groups of closed hyperbolic $3$-manifolds \cite{KM1,KM2}. Such a quasi-Fuchsian subgroup is represented by a minimal surface constructed by solving an asymptotic Plateau problem. This minimal surface may not be unique unless the limit set is an almost circular quasicircle.

The authors propose an asymptotic counting of closed quasi-Fuchsian minimal surfaces by their areas, as the limit quasicircle becomes more and more circular. They define an entropy that involves a double limit and satisfies a rigid inequality, analogous to the Besson-Courtois-Gallot's inequality. The ideas of Calegari, Marques, and Neves beautifully combine methods from topology, geometric analysis, and homogeneous dynamics (Ratner’s theory). In a recent Bourbaki seminar \cite{LabourieBourbaki}, Labourie presented these results, adhering more closely to Gromov’s philosophy of foliated Plateau problems.

\subsection{Equidistribution and asymptotic counting of $k$-surfaces}

Labourie’s presentation inspired our joint paper with Lowe and Smith \cite{ALS}, where we returned to investigating the distribution of closed $k$-surfaces, focusing specifically on closed quasi-Fuchsian $k$-surfaces. The main results are presented in \S \ref{sss.geo_counting}. We define an entropy-like functional that counts closed quasi-Fuchsian $k$-surfaces by their areas and satisfies a rigid inequality. This functional is simpler because it involves only one limit, as Plateau problems are better behaved for $k$-surfaces than for minimal surfaces.

Our methods of proof differ significantly from those in \cite{CMN} and require an extensive dynamical study of the space of $k$-surfaces. These methods involve results on solutions of different foliated Plateau problems that do not hold for minimal surfaces. Notably, we demonstrate that \emph{the unit tangent bundle of any closed, negatively curved manifold is foliated by $k$-surfaces and that this foliation is stable}. Such a result generally does not hold for minimal surfaces, as shown by Lowe in his thesis \cite{LoweGAFA}. This investigation was continued in the recent \cite{ALS2}, in which we studied other kind of spectra, in the spirit of the \emph{period spectrum} in the theory of thermodynamical formalism. This leads us to agree with Labourie that $k$-surfaces provide a convincing higher-dimensional generalization of the geodesic flow.

\subsection{Outline of the paper}

The present paper aims to provide an exposition of these results. Our goal is to explain the dynamical tools involved in proving rigidity results for both closed geodesics and closed surfaces.

In Section 2, we review useful tools from the dynamical study of periodic geodesics, focusing on their abundance and equidistribution, which are essential for studying geometric rigidity. Section 3 presents Labourie’s theory of $k$-surfaces and some results on the solution of foliated Plateau problems. Notably, Smith has written several expository papers in recent years that offer a good introduction to Labourie’s theory (see \cite{Smith_Mobius,Smith_quaternions} and references therein).  Section 4 reviews the main results of \cite{ALS}, discussing the equidistribution of closed quasi-Fuchsian k-surfaces and various rigidity results. In the final section, we pose some interesting questions for further research.

\subsection*{Acknowledgements}

{\footnotesize It is a pleasure to thank Ben Lowe and Graham Smith for all they have taught me. I am also grateful to Andrea Seppi for kindly inviting me to write this paper for the \emph{Actes du Séminaire de Théorie Spectrale et Géométrie} and for his careful reading of the manuscript. I would also like to thank Davi Obata for his comments on an earlier draft. Finally, I acknowledge financial support from CSIC, through the Grupo I+D 149 \emph{Geometría y Acciones de Grupos}, and from the IRL-2030 IFUMI, Laboratorio del Plata.}

\section{Geodesics and geometric rigidity}

\subsection{Setting} 

In this paper, $(M,h_0)$ will denote a closed, connected $3$-dimensional hyperbolic manifold. It is isometric to $\Hyp/\Pi$, where $\Pi$ is a cocompact lattice of $\PSLc=\Isom^+(\Hyp)$. By Mostow's rigidity theorem \cite{Mostow} the hyperbolic metric on $M$ is unique up to isometry.

Our main goal is to characterize the hyperbolic metric on $M$ within the space of Riemannian metrics $h$ in $M$ that have negative sectional curvature everywhere ($\sect_h<0.$)

Let $X=\tM$ denote the universal cover of $M$, and let $\bord X$ denote its ideal boundary. This boundary is defined as the set of equivalence classes of geodesic rays under the relation of ``staying at bounded distance''. The group $\Pi\simeq\pi_1(M)$ (identified with a cocompact lattice of the group of direct isometries of $X$) acts on this ideal boundary.

\subsection{The geodesic flow}

The \emph{unit tangent bundle} $T^1X$ is the set of unit vectors tangent to $X$. A vector $v\in T^1X$ directs a unique complete and oriented geodesic. Pushing $v$ along this geodesic at unit speed over time $t$ defines the \emph{geodesic flow}, denoted $G_t(v)$.

The flow lines of this process define a $1$-dimensional foliation of $T^1X$ denoted by $\ctG$. The space of leaves is the space $\bord^{(2)}X$ consisting of pairs $(\xi_-,\xi_+)$ of points in $\bord X$ such that $\xi_-\neq\xi_+$.

This flow commutes with the action of $\Pi$ on $T^1 X$ by differentials of  isometries. Therefore, it descends to a flow $g_t$ in the quotient $T^1M$. We will see that while this flow is a paradigmatic model for chaotic flows, it is independent of the negatively curved metric $h$.

\subsubsection{Boundary correspondence and Morse's lemma}\label{sss.Morse_boundary}

To prove that the flow is independent of the metric, the first step is to identify the ideal boundaries of the universal covers of $M$ for different metrics $h$ with $\sect_h<0$.

The identity map from $(M,h_0)\to(M,h)$ is bilipschitz because $M$ is compact. It has a lift $F:\Hyp\to X$ that conjugates the actions of $\Pi$. Hence the images by $F$ of geodesics in $\Hyp$ are quasi-geodesics in $X$.

\begin{theorem}[Morse's lemma]\label{th_Morse_lemma}
Quasi-geodesics in $X$ stay at bounded distance from geodesics.
\end{theorem}

Morse's Lemma provides a homeomorphism $\alpha:\C\PP^1\to\bord X$ that maps the end of a geodesic ray in $\Hyp$ to the end of its image. By construction, this homeomorphism conjugates the actions of $\Pi$. We call this map the \emph{boundary correspondence}.

\subsubsection{Hyperbolic properties and structural stability} \label{sss.hyp_properties}

Assuming the sectional curvatures of $h$ are pinched between two negative constants $-b^2<-a^2<0$, there exist two subbundles of the tangent bundle $TT^1X$, denoted by $E^s$ and $E^u$ such that
$$TT^1X=E^s\oplus E^0\oplus E^u.$$
Here $E^0$ represents the $1$-dimensional subbundle generated by the flow. These subbundles satisfy

\begin{equation}\label{eq.Anosov}
\begin{cases}
\forall  v_s\in E^s,\,\forall t>0,\,\,\frac{a}{b}e^{-bt}|v_s|\leq |DG_t(v_s)|\leq \frac{b}{a}e^{-at}|v_s|,\\
\forall  v_u\in E^u,\,\forall t>0,\,\,\frac{a}{b}e^{-bt}|v_u|\leq |DG_{-t}(v_u)|\leq \frac{b}{a}e^{-at}|v_u|.
\end{cases}
\end{equation}

The notation $|.|$ refers to the Sasaki metric (see \S \ref{sss.MAunit_tangent_bundle}). The bundles $E^s$ and $E^u$ are the bundles tangent to \emph{stable and unstable horospheres} in $X$, which are level sets of the \emph{Busemann functions} 
\begin{equation}\label{busemann}\beta_\xi(x,y)=\lim_{t\to\infty} d(y,c(t))-d(x,c(t))
\end{equation}
, where $c(t)$ is any geodesic ray parametrized by arc length and  asymptotic to $\xi\in\bord X$, endowed with inward and outward pointing unit vector fields.

Consider the flow $g_t$ in the quotient manifold $T^1M$. Property \eqref{eq.Anosov} is called the \emph{Anosov property}. It is proven in \cite{Anosov} that flows with the Anosov property are \emph{structurally stable}.
\begin{theorem}[Structural stability]
A flow in $T^1M$ that is close enough to the geodesic flow $g_t$ in the $C^1$-topology is orbit equivalent: there exists a homeomorphism $\Phi:T^1M\to T^1M$ that conjugates their (oriented) orbit foliations.

\end{theorem}

\subsubsection{Gromov's geodesic rigidity} 

In a paper from 1976 that was published in 2000 \cite{Gromov3}, Gromov went further and proved that the qualitative behaviour of geodesics in $T^1M$ is in fact independent of the negatively curved metric in $M$ (and not only locally). He proved the following theorem.

\begin{theorem}[Gromov's rigidity]\label{th_gromov_rigidity} Let $h_1$ and $h_2$ be two negatively curved metrics in $M$. Then the geodesic flows of $h_1$ and $h_2$ are orbit equivalent.
\end{theorem}

This theorem is easy in dimension $2$. Indeed let $M$ be a closed surface of genus $g\geq 2$, $h$ be a negatively curved Riemannian metric on $M$, and $X$ its universal cover (homeomorphic to a disc). Then $T^1X$ can be identified with the space $\bord^{(3)} X$ of oriented triples of points ($\xi_-,\xi_0,\xi_+)$ on the boundary. The correspondence can be described as follows. There is a unique geodesic $c$ oriented from $\xi_-$ to $\xi_+$. The orthogonal projection from $\xi_0$ to $c$ defines a vector $v$ tangent to $c$.

This correspondence $T^1X\to\bord^{(3)}X$ conjugates the action of $\Pi\simeq\pi_1(M)$ on $T^1X$ by differentials of isometries, and the diagonal action on $\bord^{(3)}$. The geodesic flow defines a action of $\R$ on $\bord^{(3)}X$, fixing the points $\xi_-$ and $\xi_+$ and moving the point $\xi_0$ from $\xi_-$ to $\xi_+$. Let $h'$ be another negatively curved metric on $M$. The boundary correspondence defined in \ref{sss.Morse_boundary} immediately defines an equivariant correspondence $\bord^{(3)}X\to\bord^{(3)}X'$ that yields the desired orbit equivalence between the geodesic flows of $h$ and $h'$.

In dimension $3$ (or higher), this is more complex. Assume first that the curvature of $h$ is constant. It is possible to associate to a vector $v\in T^1\Hyp$ a pair of points $(\xi_-,\xi_+)\in\bord^{(2)}\Hyp$. However, the role of the point $\xi_0$ is now played by a circle $c$ that separates these two points and forms the ideal boundary of the totally geodesic plane orthogonal to $v$. In this way, the unit tangent bundle $T^1\Hyp$ can be identified with the set $(c,\xi_+)$ such that $c$ is an oriented circle of $\C\PP^1$ and $\xi_+$ is a point in the interior of $c$. The point $\xi_-$ is the image by $\xi_-$ under the inversion of $c$.

If $h$ has variable curvature it is not clear how a pair $(c,\xi_+)$ defines a element of $T^1X$ because there is no reason why $c$ should span a totally geodesic plane. Gromov's strategy is to use Morse's lemma. Consider the images of geodesics of $\Hyp$ by an equivariant bilipschitz diffeomorphism $F:\Hyp\to X$. Such a curve $\gamma$ is a quasi-geodesic of $\Hyp$ that can be orthogonally projected to its geodesic escort $\gamma_0$ (see Theorem \ref{th_Morse_lemma}). This provides a surjective map $t\in\R\to y(t)\in\R$ that sends the arc length parameter on $\gamma$ onto the arc length parameter of $\gamma_0$. The problem is that it is not necessarily injective. Gromov's idea is to replace $y(t)$ with its averages over the interval $[t,T+t]$. For sufficiently large $T$, the resulting parametrization is injective. This clever trick is non-canonical, since it depends on the averaging method used.

\subsubsection{Topological entropy and asymptotic counting}\label{sss_topo_entropy}

As with every Anosov flow, the geodesic flow in $T^1M$ has positive \emph{topological entropy}, representing the exponential complexity of its orbit structure. In the context of hyperbolic flows, the works of Bowen and Margulis (see \cite{Bowen_growth,Margulis_applications}) show that this entropy equals the exponential growth rate of the set of periodic orbits. Formally it can be expressed as
$$H(M,h)=\lim_{R\to\infty}\frac{1}{R}\log\#\{\gamma\,\text{closed geodesic};\,\,\,\length_h(\gamma)\leq R\}.$$

\subsection{Rigidity results} 

We present two paradigmatic examples of rigidity results that fundamentally rely on the dynamical properties of the geodesic flow.

\subsubsection{Entropy and rigid inequality} 
The first result shows that the topological entropy of the geodesic flow satisfies a rigid inequality, interpreted as follows: when the volume of a negatively curved metric is fixed, in average, closed geodesics are shorter in variable curvature than in constant curvature.

\begin{theorem}[Besson-Courtois-Gallot, \cite{BCG}]\label{t.BCG}
Let $(M,h_0)$ be a closed hyperbolic $3$-manifold. Let $h$ be a Riemannian metric on $M$ with negative sectional curvature: $\sect_h <0$. Suppose
$$\Vol(M,h)=\Vol(M,h_0).$$
Then
$$H(M,h)\geq 2=H(M,h_0).$$
Equality holds if and only if $h$ and $h_0$ are isometrc, i.e., there exists a smooth diffeomorphism $\phi:M\to M$ such that $\phi^\ast h=h_0$.
\end{theorem}

\subsubsection{Rigidity of the hyperbolic marked length spectrum} 

The second result is an inverse problem.  Recall the definition of the marked length spectrum. Let $[\Pi]$ denote the set of conjugacy classes of $\Pi\simeq\pi_1(M)$. Given $[g]\in[\Pi]$ there exists a unique closed geodesic $\gamma_g$ representing $[g]$. The \emph{marked length spectrum}  of the metric $h$ is the map
$$\MLS_h:[\Pi]\to\R^+,\,\,\,\,\,\,\,\,\,[g]\mapsto\length_h(\gamma_g).$$

Conjecturally, the marked length spectrum determines the metric on $M$ (see \cite[Problem 3.1]{BurnsKatok}). The $2$-dimensional case was solved by Otal \cite{Otal}. In higher dimension, there are two results. A local result was obtained by Guillarmou and Lefeuvre in \cite{GuillarmouLefeuvre}. The second result, that will be relevent for our purposes, obtained by Hamenst\"adt, is the rigidity of the \emph{hyperbolic marked length spectrum}.

\begin{theorem}[Hamenst\"adt, \cite{HamenstadtMLS}]\label{th.ham_MLS} Let $(M,h_0)$ be a closed hyperbolic $3$-manifold. Let $h$ be a Riemannian metric on $M$ with negative sectional curvature: $\sect_h <0$. If $\MLS_{h}=\MLS_{h_0}$ then $h$ and $h_0$  are isometric.
\end{theorem}

\subsection{Distribution of closed geodesics}

 Both theorems use a basic fact: that the behaviour of periodic orbits of the geodesic flows determines the dynamical properties of the geodesic flow.
 
\subsubsection{Equidistribution} Given a periodic orbit of the geodesic flow $\gamma:\R\to T^1M$, the \emph{Dirac mass at $\gamma$}, denoted $\delta_\gamma$, is the invariant probability measure supported by $\gamma$. It satisfies the following property for every continuous function $f:T^1M\to\R$
$$\frac{1}{T}\int_0^T f\circ\gamma(t) dt=\int_{T^1M}fd\delta_\gamma.$$

Bowen \cite{Bowen_growth} and Margulis \cite{Margulis_measures} show that periodic measures equidistribute towards an ergodic measure $\mu_{BM}$, known as the \emph{Bowen-Margulis measure}. This is expressed as
$$\frac{1}{\#\{\length_h(\gamma)\leq R\}}\sum_{\length_h(\gamma)\leq R}\delta_\gamma\To_{R\to\infty}\mu_{BM},$$
where the convergence is in the weak-$\ast$ sense. The Bowen-Margulis measure is the measure of maximal entropy, has a nice local product structure, and coincides with the volume when the sectional curvature is constant (see \cite{Bowen_equidistribution}). The Bowen-Margulis measure $\mu_{BM}$ also has a nice description in the cover $T^1X$ that we describe below.

\subsubsection{Conformal densities}

Let $H>0$. A \emph{conformal density} of dimension $H$ is a family $(\nu_x)_{x\in X}$ of finite Borel measures on $\bord X$ that are pairwise absolutely continuous with Radon-Nikodym derivative given by
$$\frac{d\nu_y}{d\nu_x}(\xi)=\exp\left[-H\beta_\xi(x,y)\right],$$
for all $x,y\in X$ and $\xi\in\bord X$ and $\beta_\xi$ is the Busemann function at $\xi$ (see \eqref{busemann}). We say that it is invariant by the discrete subgroup $\Pi\leq\Isom^+(X)$ if for every $\gamma\in\Pi$ and $x\in X$
$$\gamma_\ast \nu_x=\nu_{\gamma x}.$$

When $H=H(M,h)$, the topological entropy, there exists a unique conformal density (up to a multiplicative constant), which is given by \emph{Patterson-Sullivan}'s construction (we will refer to the excellent monograph \cite{PPS} and references therein). This family of measures is known as the \emph{Patterson-Sullivan} measures.  The construction begins with a basepoint $o\in X$ and a family of finite measures on $\Gamma.o$ defined for all $s>H$ by
$$\nu_{s,x}=\frac{1}{\sum_{\gamma\in\Pi}e^{s\,\dist(o,\gamma o)}}\sum_{\gamma\in\Pi}e^{-s\,\dist(x,\gamma o)}\delta_{\gamma o}.$$

As $s$ tends to $H$, this family of measures above converges to the unique conformal density $(\nu_x)_{x\in X}$. An important property of this density is that measures $\nu_x$ have no atoms.

\subsubsection{Bowen-Margulis measures in Hopf's coordinates}\label{sss.BM_measures} 
The unit tangent bundle $T^1X$ can be identified with $\bord^{(2)} X\times\R$ using \emph{Hopf's coordinates}, defined as follows. Consider a base point $o\in X$. 
A unit vector $v\in T^1X$ maps to $(\xi_-,\xi_+,t)$ where $\xi_\pm$ are the endpoints of the geodesic directed by $v$ and $t=\beta_{\xi^+}(o,x)$, with $x$ being the basepoint of $v$. In these coordinates the geodesic flow acts by translation on the last factor, and the group $\Pi$ acts as follows
$$\gamma.(\xi^-,\xi^+,t)=\left(\gamma\xi^-,\gamma\xi^+,t+\beta_{\xi^+}(\gamma^{-1} o,o)\right).$$

The \emph{Bowen-Margulis measure} $\tilde{\mu}_{BM}$ is defined by
$$d\mu_{BM}(\xi^-,\xi^+,t)=\exp\left[H\beta_{\xi^-}(o,x)+H\beta_{\xi^+}(o,x)\right]d\nu_o(\xi^-)\,d\nu_o(\xi^+)\,dt.$$

This measure is independent of the basepoint $o$, invariant by the action of $\Pi$ and the geodesic flow. The quotient measure is precisely the \emph{Bowen-Margulis measure} of $T^1M$.

The measure $\exp\left[H\beta_{\xi^-}(o,x)+H\beta_{\xi^+}(o,x)\right]d\nu_o(\xi^-)\,d\nu_o(\xi^+)$ on the space $\bord^{(2)}X$ of geodesics (here $x$ is any point on the geodesic $(\xi^-,\xi^+))$, is an example of a \emph{geodesic current}.

\subsubsection{Use of conformal densities in Besson-Courtois-Gallot's theorem}
 The proof of Theorem \ref{t.BCG} uses the conformal density $(\nu_x)_{x\in X}$ to construct a \emph{natural map} $F:(X,h)\to(\Hyp,h_0)$ with the following properties
\begin{enumerate}
\item $F$ is of class $C^1$;
\item the jacobian of $F$ satisfies
$$|\textrm{Jac} F|\leq\left(\frac{H(M,h)}{H(M,h_0)}\right)^3,$$
throughout $X$.
\item if at a point $y\in X$
$$|\textrm{Jac} F(y)| = \left(\frac{H(M,h)}{H(M,h_0)}\right)^3,$$
then $DF_y$ is a homothety with ratio $H(M,h)/H(M,h_0)$.
\end{enumerate}

The map $F$ is defined as follows: for a point $x\in X$, we define a measure on the Riemann sphere $\C\PP^1$ by the formula
$$m_x=\alpha_\ast\nu_x,$$
where $(\nu_x)_{x\in X}$ is the conformal density of dimension $H$ and $\alpha:\bord X\to\C\PP^1$ is the boundary correspondence. The measure $m_x$ has no atom, so it has a well-defined \emph{barycenter}
$$z=\textrm{bar}(m_x),$$
defined as the point in $\Hyp$ that minimizes the integral $\int_{\C\PP^1}\beta_{\xi}(o,z)\,dm_x(z).$ The natural map is then given by the correspondence $F:x\in X\mapsto z=\textrm{bar}(\alpha_\ast \nu_x)\in\Hyp.$ Remarkably, this map, which involves the family of measures describing the distribution of closed geodesics, satisfies the properties above.

\subsubsection{Rigidity of the hyperbolic marked length spectrum} 

Hämenstadt's result also relies on the idea that the distribution of closed geodesics determines the dynamics of the geodesic flow. Cohomological methods prove that the equality of marked length spectra implies the existence of a \emph{time-preserving} conjugacy of the geodesic flows. This is contained in a theorem proven by Guillemin and Kazhdan in \cite{Guillemin_Kazhdan}.

\begin{theorem}[Guillemin-Kazhdan]\label{th.cohomology}
Let $h_1$ and $h_2$ be two negatively curved metrics on $M$ that have the same marked length spectra. Let $(g_t^{(i)})_{t\in\R}$ denote the geodesic flow of $(M,h_i)$. Then there exists a homeomrphism $\Phi:T^1M\to T^1M$ such that for every $t\in\R$
$$\Phi\circ g_t^{(1)}=g_t^{(2)}\circ\Phi.$$
\end{theorem}

\begin{proof}[Sketch of proof]
Fix an orbit equivalence between the geodesic flows $\phi$ which we may assume is Hölder continuous (together with the boundary correspondence) and smooth along the flow. This regularity is required to apply Liv\v{s}ic's theorem (see \cite{Livsic}), which is a fundamental tool in the study of Anosov flows. Let $X^{(i)}$ be the generator of the geodesic flow $g_t^{(i)}$. Since $\phi$ is Hölder continuous and smooth along the flow, there exists a Hölder continuous function $\lambda:T^1M\to T^1M$ such that for every $v\in T^1M$
$$\mathcal{L}_{X^{(1)}}\phi(v)=\lambda(v) X^{(2)}(v).$$
The equality between marked length spectra implies that for every $v\in T^1M$ whose trajectory is periodic of period $T$, we have
$$\int_0^T \lambda\left(g_t^{(1)}(v)\right)dt=T=\int_0^T dt,$$
so if $\alpha=\lambda-1$, which is Hölder continuous, we have that for every $v\in T^1M$ whose trajectory is periodic of period $T$
$$\int_0^T \alpha\left(g_t^{(1)}(v)\right)dt=0.$$

Liv\v{s}ic's theorem implies that $\alpha$ is a coboundary, meaning that 
$$\alpha=\mathcal{L}_{X^{(1)}}A,$$
for some Hölder continuous function $A$, smooth along the flow. Then the map
$$\Phi(v)=g^{(2)}_{-A(v)}\circ \phi(v)$$
satisfies the desired property.
\end{proof}

If the map given by Kazhdan-Guillemin were of class $C^1$ then H\"amenstadt's theorem \ref{th.ham_MLS} would be a consequence of Besson-Courtois-Gallot's result. Indeed, recall that $h_0$ denotes the hyperbolic metric of $M$ and that $h$ is a Riemannian metric on $M$ with $\sect_h<0$. Assume that there exists a $C^1$ diffeomorphism $\Phi$ conjugating the geodesic flows $g_t^0$ and $g_t$ of $h_0$ and $h$ respectively. Then the geodesic flows have the same topological entropies (see \S \ref{sss_topo_entropy}). If we can prove that $(M,h)$ and $(M,h_0)$ have the same volume then Besson-Courtois-Gallot's theorem implies that they are isometric.

To see this, recall that $g_t$ is a contact Anosov flow: there exists a $1$-form $\lambda$, called the canonical form of $g_t$, such  that $\lambda(X)=1$ (with $X$ representing the generator of $g_t$) and that vanishes in $E^s\oplus E^u$ (where $E^s$ and $E^u$ are the stable and unstable bundles defined in \S \ref{sss.hyp_properties}). Note that $\lambda\wedge d\lambda^2$ is the Liouville measure (the volume form) in $T^1M$. Hence if $\Phi$ conjugates $g_t$ and $g_t^0$ then $\Phi^\ast(\lambda^0)=\lambda$ (where $\lambda^0$ is the canonical $1$-form of $g_t^0$), which means that $\Phi$ preserves the volume forms of $T^1M$, conluding the proof.

However, the conjugacy given by Kazhdan-Guillemin is in general only continuous, and there is no way to pull back of the canonical $1$-form. Hamenst\"adt found a clever argument, based on thermodynamical formalism and the theory of H\"older cocycles and geodesic currents (see \cite{Hamenstadt_cocycles_hausdorff,Ledrappier_bord}), to prove directly that the time-preserving conjugacy must send the Liouville measure of $h$ onto that of $h_0$. This uses that the Anosov splitting of the geodesic flow of $g_t^0$ is of class $C^1$.

\section{Foliated Plateau problems and dynamics of $k$-surfaces}

\subsection{Asymptotic Plateau problems for $k$-surfaces}

\subsubsection{$k$-surfaces} Recall that $(M,h)$ denotes a closed and oriented $3$-manifold such that $\sect_h\leq -1$ and that $X$ denotes its universal cover. Let $S$ be an oriented surface and $e:S\to X$ be a smooth immersion into $X$. Let $N$ be the orthogonal unit vector field compatible with the orientation, and let $A(u)=\nabla_uN$ be the \emph{shape operator}. The \emph{extrinsic curvature} is defined as
\begin{equation}\label{eq.ext_curvature}
\kappa_{ext}=\det(A).
\end{equation}

\emph{Gauss' equation} relates the \emph{intrinsic} and \emph{extrinsic} curvatures of $e(S)\dans X$ by
\begin{equation}\label{eq.Gauss_eq}
\kappa_{int}=\kappa_{ext}+\sect_h|_{Te(S)}.
\end{equation}

The \emph{Gauss lift} of $S$ is defined as the surface $\hat S=\hat e(S)\dans T^1X$: this is the surface $S$ attached with the unit tangent vector field $N$.

\begin{defi} Let $k\in(0,1)$. We define a \emph{$k$-surface} as a complete and smooth immersed surface that has constant extrinsic curvature
$$\kappa_{ext}=k.$$
\end{defi}

\begin{remark}
Labourie's original definition requires that the Gauss lift $\hat S\dans T^1X$ be complete, rather than $S$ itself. This is coherent with his compactness result (see Theorem \ref{th.Labourie_compactness}).
\end{remark}
By Gauss' equation, if $\sect_h\leq -1$ everywhere and $0<k<1$ we have at all points of a $k$-surface
\begin{equation}\label{eq.Gauss_eq_ksurf}
\kappa_{int}\leq k-1<0.
\end{equation}

\subsubsection{Asymptotic Plateau problem}

We formulate the asymptotic Plateau problem for $k$-surfaces inside $(X,h)$. Let $c\dans\partial_\infty X$ be an Jordan curve. We provide $c$ with an \emph{orientation}, that is, we distinguish a connected component of $\partial_\infty X\moins c$ that we call the \emph{exterior component} and denote by $\Ext(c)$. This will be thought of as our boundary condition. The other connected component of $\bord X\moins c$ is called the \emph{interior component} and denoted by $\mathrm{Int}(c)$.

Fix $k\in (0,1)$. The \emph{asymptotic Plateau problem} that we want to solve is: find $D\dans X$ such that
\begin{enumerate}
\item $D$ is a complete embedded and oriented disc in $X$;
\item $\partial_\infty D=c$, where $\bord D$ is endowed with the orientation it inherits from $D$; and
\item $D$ has constant extrinsic curvature: $\kappa_{ext}=k$ everywhere.
\end{enumerate}

When the properties above are satisfied, we say that $c$ \emph{spans} the $k$-disc $D$. This is a particular case of the asymptotic Plateau problems considered by Labourie \cite{LabourieInvent} (see also Smith \cite{Smith_asymp}). In these references, the authors define more general boundary conditions and characterize which asymptotic Plateau problems have solution (which is automatically unique) and which do not. See, for example, Corollary \ref{coro.plateau_no_sol} for asymptotic problems without a solution.

The following result is a consequence of \cite{LabourieInvent} and \cite{Smith_asymp}. When $X=\Hyp$ it was proven by Rosenberg-Spruck in \cite{Rosenberg_Spruck}.

\begin{theorem}[Asymptotic Plateau problem]\label{th_as_plateau_pb}
For all $0<k<1$ and every oriented Jordan curve $c\dans\bord X$ there exists a unique $k$-disc $D_c$ spanned by $c$. Furthermore, there exists a convex subset $K_c\dans X$ such that 
\begin{enumerate}
\item $\partial K_c=D_c$;
\item for every $\xi\in\mathrm{Int}(c)$, there exists a geodesic ray $\gamma$ included in $(K_c)^c$ and normal to $D_c$ such that $\gamma(\infty)=\xi$.
\end{enumerate}
\end{theorem}

The analogous problem for minimal surfaces instead of $k$-surfaces, in particular the existence and uniqueness of solutions in hyperbolic $3$-space, was studied by Anderson \cite{Anderson} and Uhlenbeck \cite{Uhlenbeck} in the 1980s.

\subsection{Monge-Ampère equation and $k$-surfaces}

\subsubsection{$J$-holomorphic curves}

\label{sss.J-holo} Consider the space $W=\C\oplus\C=\R^2\oplus\R^2$. Denote $V=\R^2\oplus\{0\}$ and $H=\{0\}\oplus\R^2$. Define the complex structure on $W$ by
\begin{equation}\label{eq.J}
J=\begin{pmatrix}
0    &  J_0  \\
J_0  &  0
\end{pmatrix}
\end{equation}
where
\begin{equation}\label{eq.J_0}
J_0=\begin{pmatrix}
0    &  -1  \\
1  &  0
\end{pmatrix}
\end{equation}
represents the usual multiplication by $i$ of $\C$. A \emph{$J$-holomorphic} line of $W$ is a $2$-dimensional subspace $L$ which is $J$-invariant.

We say that a $2$-dimensional submanifold $S\dans W$ is a \emph{$J$-holomorphic curve} if the tangent plane at each point of $S$ is $J$-holomorphic.

When $L$ is the graph $\{(v,Av):v\in V\}$ of a $2\times 2$ matrix $A$, then $L$ is $J$-holomorphic line if and only if
$$AJ_0A=J_0,$$
which is equivalent to
$$A=A^T\,\,\,\,\,\,\,\,\text{and}\,\,\,\,\,\,\,\,\det(A)=1.$$

If we modify the complex structure by setting for some $k>0$
\begin{equation}\label{eq.Jk}
J=\begin{pmatrix}
0    &  \sqrt{k}^{-1}J_0  \\
\sqrt{k}J_0  &  0
\end{pmatrix},
\end{equation}
then the graph of $A$ is $J$-holomorphic if and only if

$$A=A^T\,\,\,\,\,\,\,\,\text{and}\,\,\,\,\,\,\,\,\det(A)=k.$$

The tangent plane at $(x_0,v(x_0))$ of the graph $\{(x,v(x)):x\in\Omega\}$ of a smooth map $v:\Omega\dans V\to\R^2$ is the graph of the linear map $Dv(x_0)$. So this graph is \emph{$J$-holomorphic}, where $J$ is given by \eqref{eq.J} (resp. \eqref{eq.Jk}), if and only if at each point $x_0\in\Omega$ the derivative $Dv(x_0)$ is symmetric and satisfies $\det(Dv(x_0))=1$ (resp. $\det(Dv(x_0))=k$).

\subsubsection{Monge-Ampère's equation}

Let $k>0$. We consider the \emph{Monge-Ampère equation} given by
\begin{equation}\det(D^2u)=k,
\end{equation}
where $u:\Omega\to\R$ is a function of class $C^2$ defined in an open set $\Omega\dans\R^2$.
Labourie's theory of $k$-surfaces relies on the following observation which is a simple consequence of \S \ref{sss.J-holo}.

\begin{theorem}\label{th.Jholo=MA}
A function $u:\Omega\dans\R^2\to\R$ of class $C^2$ is a solution of the Monge-Ampère equation if and only if the graph of the derivative map $Du:\Omega\to\R^2$ is a $J$-holomorphic curve.
\end{theorem}

\subsubsection{Monge-Ampère structures on the unit tangent bundle}\label{sss.MAunit_tangent_bundle} 

Recall that the tangent bundle to $T^1X$ has a natural splitting 
$$T_\xi T^1X=H_\xi\oplus\xi^\perp,$$
where $\xi\in T^1X$. Here $H_\xi$ represents the \emph{horizontal distribution} defined by parallel transport of unit vector fields along curves in $X$; and $\xi^\perp$ is the tangent plane at $\xi$ to the unit sphere $T^1_xX$ (where $x$ denotes the basepoint of $\xi$).

The \emph{Sasaki metric} in $T^1X$ makes this splitting orthogonal and induces the metric $h$ in both $H_\xi$ and $\xi^\perp$.

Let $\hat S$ be the Gauss lift of a smoothly immersed and complete surface $S$. In the coordinates given above, the tangent plane at $\xi\in\hat S$ is the graph of the shape operator:
$$T_\xi\hat S=\{(v,Av): v\in\xi^\perp\}\dans W_\xi:=\xi^\perp\oplus\xi^\perp.$$

The distribution $W$ defined above (by $W_\xi=\xi^\perp\oplus\xi^\perp$) is a $4$-dimensional subbundle of $T T^1X$ that we call the \emph{holonomy bundle}.

The metric and the orientation on the $3$-dimensional manifold $X$ provide $TX$ with a wedge product $\wedge$. The plane $\xi^\perp$ then comes with a natural orientation and a canonical complex structure $j_\xi$ defined for $\nu\in\xi^\perp$ by
$$j_\xi.\eta=\xi\wedge\eta.$$
It is clear that $j_\xi^2=-\Id$ on $\xi^\perp$. Given $k\in(0,1)$, it is then possible to define the linear transform of $W_\xi$ that in the splitting $\xi^\perp\oplus\xi^\perp$ reads
\begin{equation}\label{eq.k_MAstructure}J_\xi=\begin{pmatrix}
0    &  \sqrt{k}^{-1}j_\xi  \\
\sqrt{k}j_\xi  &  0
\end{pmatrix}.
\end{equation}

This defines a \emph{complex structure} on $W$, that is a section $J$ of $\text{End}(W)$ such that $J^2=-\Id$.  As in \S \ref{sss.J-holo}, the condition of being a $k$-surface, that is the fact that $\det(A)$ is constant equal to $k$ (recall that the shape operator $A$ is symmetric), can be expressed as follows

\begin{enumerate}
\item $T\hat S$ is tangent to $W$;
\item $\hat S$ is a $J$-holomorphic curve meaning that $T\hat S$ is stable by $J$, where $J$ is given by \eqref{eq.k_MAstructure}; and 
\item $T\hat S$ is transverse to $V=\{0\}\oplus\xi^\perp$.
\end{enumerate}

Complete surfaces $\hat S\dans T^1X$ satisfying items (1) and (2) above are called \emph{Monge-Ampère surfaces}. There exist other examples of Monge-Ampère surfaces, which don't satisfy (3) and can be obtained as (finite coverings of) the normal bundles $\text{N}\gamma$ of complete geodesics $\gamma$. Such a surface is \emph{everywhere tangent} to $V$ and is called a \emph{tube}.

\subsubsection{Labourie's compactness theorem}

Using elliptic properties of Monge-Ampère equations, Labourie proves the following compactness theorem in \cite{LabourieGAFA}. The convergence of  sequences of \emph{pointed complete Riemannian surfaces immersed in $T^1X$} $(\hat S,p,\hat e)$ (here $p\in\hat S$ and $\hat e:\hat S\to T^1X$ is the immersion) is to be understood in the sense of \emph{Cheeger-Gromov}. We refer to \cite{AlvarezSmith,Smith_asymp} for more details.

\begin{theorem}[Labourie's compactness theorem]\label{th.Labourie_compactness}
Let $(\hat S_n,p_n,\hat e_n)$ be a sequence of Gauss lifts to $T^1X$ of $k$-surfaces such that $\hat e_n(p_n)$ is precompact in $T^1X$. Then the $(\hat S_n,p_n,\hat e_n)$ is precompact in the Cheeger-Gromov sense. Furthermore, any accumulation point is either the Gauss lift of an immersed $k$-surface or a tube.
\end{theorem}

\subsection{Geometric maximum principles and continuity of $k$-surfaces}

Convexity and elliptic properties of Monge-Ampère equations are the principal ingredients in the geometric studies of $k$-surfaces. 
\subsubsection{Local geometric maximum principle}

We begin with an elementary local result regarding the behavior of convex surfaces.

An oriented surface $S$ immersed in $X$ is called \emph{convex} if its second fundamental form is positive definite at every point. Let $N$ be a unit vector field orienting $S$. We say that another convex surface $S'$ is an exterior tangent to $S$ at a point $p\in S\cap S'$ if there exists a neighbourhood $U\dans S$  of $p$ and a smooth function $\lambda:U\to[0,\infty)$ such that, for every $x\in U$, $\exp(\lambda(x)N(x))\in S'$. We also say that $S$ is an interior tangent to $S'$.

\begin{theorem}[Local geometric maximum principle]\label{th.local_max-principle}
Let $S$ and $S'$ be two convex oriented surfaces immersed in $(X,h)$. Assume that $S'$ is an exterior tangent to $S$ at a point $p\in S\cap S'$. Then
$$\kappa_{ext}^{S'}(p)\leq \kappa_{ext}^S(p).$$
\end{theorem}

A proof of this theorem can be found in \cite[Lemme 2.5.1]{LabourieInvent} and \cite[Lemma 2.3.2]{ALS}. As a consequence of this theorem, the following can be proven (see \cite[Propositions 2.5.2 and 2.5.3]{LabourieInvent}).

\begin{corollary}\label{coro.plateau_no_sol} Assume $\sect_h\leq -1$ and $0<k<1$.
\begin{enumerate}
\item There is no closed $k$-surface in $X$.
\item There is no \emph{pseudo-horospheric} $k$-surface in $X$, i.e., a $k$-surface $S\dans X$, oriented by the unit normal vector field $N$, whose image of the horizon map
$$\textrm{hor}:x\in S\mapsto \exp(\infty N(x))\in\partial_\infty X,$$
is the complement of a singleton in $\partial_\infty X$.
\end{enumerate}
\end{corollary}

The proof of this result uses the fact that when $\sect_h\leq -1$, spheres and horospheres have extrinsic curvatures $\geq 1$ at every point (see \cite[Corollaire 3.1.2]{LabourieInvent}). Hence a $k$ surface cannot be an interior tangent to a sphere or a horosphere. Therefore, Theorem \ref{th.local_max-principle} immediately implies the first item. The proof of the second item is not as direct since it is not clear that a \emph{pseudo-horosphere} should be interior tangent to a horosphere.  But Labourie proves that the existence of such a $k$-surface implies the existence of another $k$-surface which is interior tangent to some horosphere, which is a contradiction: see \cite[Lemme 2.4.1]{LabourieInvent}.

\subsubsection{Global geometric maximum principle} 
The local maximum principle has significant global consequences in our context.

\begin{theorem}[Global geometric maximum principle]\label{th.global_max}
Let $c$ and $c'$ be two oriented Jordan curves of $\partial_\infty X$. Let $D_c$ and $D_{c'}$ be the $k$-discs they span, and $K_c$ and $K_{c'}$ be the convex sets bounded by these discs (see Theorem \ref{th_as_plateau_pb}). Suppose that $\Ext(c)\dans\Ext(c')$. Then
$$K_c\dans K_{c'}.$$
\end{theorem}

The idea is that the condition $\Ext(c)\dans\Ext(c')$ implies that there exists a relatively compact open set $\Omega$ such that the inclusion $K_c\moins\Omega\dans \text{Int}(K_{c'})\moins\Omega$ holds. Now if the inclusion does not hold true inside $\Omega$, there will be a tangency point between $D_{c}$ and some equidistant to $D_c'$ or with some cylinder over some geodesic (see the proof of \cite[Theorem 2.3.1]{ALS}) which contradicts the local geometric maximum principle (equidistant to a $k$-surface and cylinder both have extrinsic curvatures $>k$). We refer to \cite[Lemmas 2.3.5 and 2.3.6]{ALS} for detailed proofs.

\subsubsection{Strong geometric maximum principle}

Using the fact that $k$-surfaces are solutions of some elliptic problem (see the previous paragraph), we can also prove a strong geometric maximum principle.

\begin{theorem}[Strong geometric maximum principle]\label{th.strong_max_pcpl}
Let $S$ and $S'$ be two $k$ surfaces in $X$. Assume that $S'$ is an exterior tangent to $S$. Then $S=S'$.
\end{theorem}

This result can be proven using a maximum principle for elliptic problems. For example in  \cite{ALS}, it is shown that it follows by an application of Aronszajn's unique continuation principle \cite{Aronszajn}. Indeed this principle implies the following result (see \cite[Lemma 2.5.4]{ALS}).

\begin{prop}\label{prop.harmonic_pol}
Consider a $k$-surface $S\dans X$ and assume that $S'$ can be written as the graph of a smooth function $f$ over an open set $U\dans S$. That is $S'=\iota(S)$ where
$$\iota:x\in U\mapsto\exp(f(x)N(x)).$$

If $\kappa_{ext}^{S'}=k$ everywhere, we have the following dichotomy:
\begin{itemize}
\item either $f=0$ and $S'\dans S$;
\item or $f$ has nonvanishing Taylor series at all points of $U$ and the lower order term of $f$ is, up to precomposition by a linear map, a harmonic polynomial.
\end{itemize}
\end{prop}

The strong maximum principle follows directly from this proposition. Indeed, suppose that $S$ and $S'$ are $k$-surfaces such that $S'$ is an exterior tangent to $S$ at a point $x$. Then $S'$ can be written as a graph of a function $f$ over $S$ in a neighbourhood of $x$. Since $S'$ is an exterior tangent to $S$ at $x$, we must have $f\geq 0$. By Proposition \ref{prop.harmonic_pol}, if $f$ is not identically $0$ then it behaves close to $x$ as a harmonic polynomial (after a linear change of coordinates), which is absurd since harmonic polynomials change sign.

\subsubsection{Continuity of the solutions of the asymptotic Plateau problem}

We endow the set of oriented Jordan curves of $\bord X$ with the Hausdorff topology. We identify the convex set $K_c$ and its closure in $\overline{X}=X\cup\bord X$ so it becomes compact. If $K$ is a compact convex subset of $\overline{X}$ we let $\bord K$ denote $K\cap\bord X$. In particular, for every Jordan curve $c\dans\bord X$,
$$\bord K_c=\overline{\Ext(c)}.$$
Let $\hat \partial K$ denote $\partial K\cup\bord K\dans\widehat{X}$. The first lemma is proven in \cite[2.2.2]{ALS}.

\begin{lemma}[Continuity of the boundary]\label{lem_cont_bdr}
The map $\hat\partial:K\mapsto\hat \partial K$, defined over the space of compact, convex subsets of $\overline{X}$, is continuous for the Hausdorff topology.
\end{lemma}

The next result is a consequence of Labourie's compactness Theorem \ref{th.Labourie_compactness}, which describes the possible ways $k$ surfaces can degenerate.

\begin{theorem}\label{th.cont_Laboubou}
Let $K\dans\overline{X}$ be the Hausdorff limit of a sequence $K_m$ of compact convex subsets of $\overline{X}$ with non-empty interiors whose boundaries $\partial K_m\dans X$ are smooth and have extrinsic curvature $k$. Then three possibilities can occur
\begin{enumerate}
\item $K$ has non-empty interior and $\partial K$ is smooth with extrinsic curvature $k$. In that case $\partial K_m$ converges to $\partial K$ in the $C^\infty_{loc}$-sense;
\item $K$ is a single point of $\bord X$;
\item $K$ is a complete geodesic of $\overline{X}$.
\end{enumerate}
\end{theorem}

\begin{theorem}[Continuity of the Plateau solution]\label{th.cont.Plateau}
The function $c\mapsto K_c$ from the space of oriented Jordan curves to the set of compact convex subsets of $\overline{X}$ (both endowed with the Hausdorff topology) is continuous. Moreover the function $c\mapsto D_c$ is continuous in the $C_{loc}^\infty$ sense.
\end{theorem}

\begin{proof}
Let $(c_n)$ be a sequence of Jordan curves converging to $c$ in the Hausdorff sense. Let $c^+$ and $c^-$ be two Jordan curves such that $c^-\dans\Ext(c)$ and $c\dans\Ext(c^+)$. So, in particular, when $n$ is large enough, $c^-\dans\Ext(c_n)$ and $c_n\dans\Ext(c^+)$. Using the global geometric maximum principle, we have that $K_{c^-}\dans K_{c_n}\dans K_{c^+}$ for all such $n$. So any accumulation point $K$ of the sequence $K_{c_n}$ satisfies $K_{c^-}\dans K\dans K_{c^+}$. Using Lemma \ref{lem_cont_bdr}, we have $\overline{\Ext(c^-)}=\bord K_{c^-}\dans\bord K\dans \bord K_{c^+}=\overline{\Ext(c^+)}$. Since $c^{\pm}$ are arbitrary, this implies that $\bord K=\overline{\Ext(c)}$. By Theorem \ref{th.cont_Laboubou}, $\partial K$ is a $k$-surface, so $K=K_c$ by uniqueness of the asymptotic Plateau problem.
\end{proof}

\subsection{Foliated Plateau problems}

We now discuss the dynamical properties of several spaces that can be seen as $2$-dimensional analogues of the geodesic flow, which aligns with Gromov’s theory of foliated Plateau problems : \cite{Gromov_FolPlateau1,Gromov_FolPlateau2}.

\subsubsection{Labourie's phase space}\label{sss.Lab_lam} Recall that $(M,h_0)$ is a closed hyperbolic $3$-manifold, that $h$ is a Riemannian manifold on $M$ with $\sect_h\leq -1$ and $X$ denotes the universal cover of $M$. We fix a number $k\in(0,1)$.

The phase space of the geodesic equation consists of trajectories with initial conditions. This is the space of all pointed geodesics, also identified with the unit tangent bundle of the manifold.

Analogously, we consider the space $\mathcal{S}_k(X)$ of all pointed immersed $k$-surfaces $(S,p)$. Let $\mathcal{T}(X)$ be the set of tubes. By Labourie's compactness theorem, the space
$$\overline{\mathcal{S}_k(X)}=\mathcal{S}_k(X)\cup\mathcal{T}(X),$$
is closed for the Cheeger-Gromov topology, inside the space of all pointed immersed surfaces inside the unit tangent bundle of $X$. Labourie proves that it comes with
\begin{enumerate}
\item \emph{a structure of lamination $\cL$}: two pointed surfaces $(S,x)$ and $(S',x')$ are in the same leaf if $S=S'$;
\item \emph{a horizon map} $\Phi:\overline{\mathcal{S}_k(X)}\to\bord X$ defined by 
$$\Phi(\hat S,v)=\exp(\infty.v)\in\bord X.$$
where $\hat S\dans T^1X$ is the Gauss lift of a $k$ surface or is a tube, and $v\in\hat S$
\item \emph{an action of $\Pi\simeq\pi_1(M)$} that preserves the lamination $\cL$ and is cocompact.
\end{enumerate}

Let $\overline{\mathcal{S}_k(M)}=\overline{\mathcal{S}_k(X)}/\Pi$. This is a compact space endowed with a lamination by $k$-surfaces denoted by $\cL$.

There is a stability result, analogous to Gromov's geodesic rigidity theorem. The local stability theorem was obtained by Labourie in \cite{LabourieInvent} and the global result stated below was recently obtained by Smith in \cite{Smith_asymp}. 

\begin{theorem}[Stability of the $k$-surface lamination]
Let $\alpha:\C\PP^1\to\bord X$ denote the boundary correspondence. Let $\cL_0$ and $\cL$ be the laminations, and $\Phi_0:\overline{\mathcal{S}_k(\Hyp)}\to\C\PP^1$ and $\Phi:\overline{\mathcal{S}_k(X)}\to\bord X$ be the horizon maps described above. Then there exists a unique homeomorphism
$$\Psi:\overline{\mathcal{S}_k(\Hyp)}\to\overline{\mathcal{S}_k(X)}$$
that conjugates the laminations $\cL_0$ and $\cL$ such that $\Phi\circ\Psi=\Phi_0\circ\alpha$. Moreover, $\Psi$ maps $\mathcal{T}(\Hyp)$ to $\mathcal{T}(X)$.
\end{theorem}

This space shares remarkable dynamical properties with the geodesic flow. 

\begin{theorem}[Dynamical properties of the $k$-surface lamination]
The following properties hold
\begin{enumerate}
\item the generic leaf of $\cL$ is dense in $\overline{\mathcal{S}_k(M)}$ \cite{LabourieInvent};
\item compact leaves of $\cL$ are dense in $\overline{\mathcal{S}_k(M)}$ \cite{LabourieInvent}; 
\item if $h$ can be deformed through negatively curved metrics to $h_0$ (this assumption may be dropped after Smith's global stability result: \cite{Smith_asymp}), then there are infinitely many probability measures in $\overline{\mathcal{S}_k(M)}$ that are totally invariant for the lamination $\cL$, ergodic, pairwise mutually singular, fully supported and obtained by a coding procedure \cite{LabourieAnnals}.
\end{enumerate}
\end{theorem}

This results are nicely presented in \cite{Labourie_phase_space}. For the purposes of \cite{ALS}, two successive refinements of Labourie’s laminated structure are considered, restricting to progressively smaller subsets of $\overline{\mathcal{S}_k(X)}$. We follow the presentation of \cite{ALS2}.

\subsubsection{Fibered Plateau problems and action of $\PSL$}\label{sss.fibered} We first consider the space of marked $k$-discs in $X$ spanned by a \emph{quasicircle}. 

Recall that given $C>1$ a $C$-quasicircle of $\C\PP^1$ is a Jordan curve $\Lambda$ that is the image of the real projective line $\Lambda=h(\R\PP^1)$ by a $C$-quasiconformal homeomorphism $h:\C\PP^1\to\C\PP^1$. We refer to Ahlfors' book \cite{Ahlfors} for definitions and properties of quasiconformal mappings. The relevant properties for us are
\begin{enumerate}
\item the image of a $C$-quasicircle by a Möbius transformation is still a $C$-quasicircle;
\item $1$-quasiconformal maps are Möbius transformations and $1$-quasicircles are round circles of $\C\PP^1$;
\item a subset of $C$-quasicircles, closed for the Hausdorff topology, is compact if and only if it does not accumumlate on a singleton.
\end{enumerate}

Given the boundary correspondence $\alpha:\C\PP^1\to\bord X$, we will say that
a $C$-quasicircle (resp. a round circle) of $\bord X$ is the image by $\alpha$ of a $C$-quasicircle (resp. a round circle) of $\C\PP^1$. We let $\text{QC}^+(C)$ denote the space of oriented $C$-quasicircles. Note that by equivariance of the boundary correspondence, $\Pi$ acts on the spaces $\text{QC}^+(C)$ for all $C$.

We now introduce various spaces, which we view as restrictions of Labourie's phase space.

\begin{itemize}
\item The space $\FKD_{k,h}(C)$, formed of all marked frame bundles $(FD,\xi)$, where $D$ is an oriented $k$-disk spanned by a $C$-quasicircle and $\xi\in FD$. We furnish this space with the $C^\infty_\loc$-topology.
\item The space $\MKD_{k,h}(C)$, formed of all marked Gauss lifts $(\hat D,n)$, where $D$ is an oriented $k$-disk spanned by a $C$-quasicircle and $n\in\hat D$. We likewise furnish this space with the $C^\infty_\loc$-topology.
\end{itemize}
$\Pi$ acts on these spaces cocompactly on the left (see \cite{ALS,LabourieGAFA}), and the canonical projection $p:\FKD_{k,h}(C)\to\MKD_{k,h}(C)$ is continuous, surjective, and $\Pi$-equivariant.

We have natural boundary maps which are continuous (see for example Lemma \ref{lem_cont_bdr})
\begin{align*}
&\bord:\FKD_{k,h}(C)\to\QC^+(C);(FD,\xi)\mapsto\bord D\ \text{and}\\
&\bord:\MKD_{k,h}(C)\to\QC^+(C);(\hat{D},n)\mapsto\bord D\ .
\end{align*}

A $k$-disc of $X$ has negative intrinsic curvature so it is uniformized by the upper half plane $\H2$. Therefore, the space $\text{MKD}_{k,h}(C)$ has the structure of a hyperbolic Riemann surface lamination (the leaves being preimages of $\bord$),  which descends to a compact hyperbolic Riemann surface lamination with total space $\text{MKD}_{k,h}(C)/\Pi$: see \cite{LabourieGAFA}. These laminations will be denoted by $\cL_{k,h}$ (both on $\MKD_{k,h}(C)$ and on the quotient). The sphere bundle of this lamination (i.e. the set of unit vectors tangent to leaves) is identified to $\FKD_{k,c}(C)$. We let $F\cL_{k,h}$ denote the natural lamination (whose leaves are preimages of $\bord$ in $\FKD_{k,h}(C)$).

Now, given a frame $\xi=(n(x),e_1,e_2)\in FD$ based at $x\in  D$, there exists a \emph{unique} uniformization map $u_\xi:\H2\to D$ such that 
$$u_\xi(i)=x,\,\,\,\,\text{and}\,\,\,\,Du_\xi(e_1^0)\propto e_1$$
where $i=\sqrt{-1}\in\H2$ and $(e_1^0,e_2^0)$ denotes the canonical basis of the plane.

We can define a \emph{right action} of $\PSL$ on $\FKD_{k,h}(C)$ as follows. Let $(FD,\xi)\in\FKD_{k,h}(C)$ and $g\in\PSL$. Write $\xi=(x,n(x)e_1,e_2)$. Consider the, uniformization map $u_\xi:\H2\to D$ defined in the previous paragraph. Then we define $(FD,\xi).g=(FD,\xi')$ where
$$\xi'=\left(u_\xi\circ g(i),n(u_{\xi}\circ g(i)),D(u_{\xi}\circ g)e_1^0,D(u_{\xi}\circ g)e_2^0\right).$$
This action, denoted by $\alpha_k:\PSL\acts\FKD_{k,h}(C)$, is free and transitive on fibers of $\bord$, and obviously commutes with the left and cocompact $\Pi$-action on $\FKD_{k,h}(C)$. The continuity of this action comes from Candel's simultaneous uniformization theorem \cite{Candel} (more precisely, see the version proven in \cite[Theorem 2.7]{AlvarezSmith}).

Using this action of $\PSL$ on $\FKD_{k,h}(C)$, in \cite[Theorems 3.1.2. and 3.2.2]{ALS} we proved the following result.

\begin{theorem}[Fibered Plateau problem]\label{t.fibered_PP}
For all $C\geq 1$, the map
\begin{equation*}
\bord :\FKD_{k,h}(C)\to\mathrm{QC}^+(C)\ ,
\end{equation*}
is a continuous principal $\PSL$-bundle over $\QC^+(C)$, and the map
\begin{equation*}
\bord :\mathrm{MKD}_{k,h}(C)\to\mathrm{QC}^+(C)\ ,
\end{equation*}
is a continuous disk-bundle, which is an associated bundle of $\FKD_{k,h}(C)$.
\end{theorem}

\subsubsection{Foliated Plateau problem}\label{sss.FPP}

Our second refinement yields a space of marked $k$-surfaces carrying a natural \emph{foliated structure}: we consider now the case where $C=1$. Let $\cC^+=\text{QC}^+(1)$ denote the set of oriented round circles in $\cC^+$. In this case, the spaces $\FKD_{k,h}(1)$ and $\MKD_{k,h}(1)$ identify respectively with the frame bundle $FX$ and the unit tangent bundle $T^1X$, as the following theorem shows.

\begin{theorem}[Foliated Plateau problem]\label{t.foliated_PP} The Gauss lifts of oriented $k$-disks spanned by round circles of $\cC^+$ form a smooth foliation $\cF_{k,h}$ of $T^1X$, and their frames form a smooth foliation $F\cF_{k,h}$ of $FX$. Furthermore, these foliations are $\Pi$-invariant.
\end{theorem}

The situation is different for minimal surfaces, as Lowe's work \cite{LoweGAFA} provides examples of a metric $h$ with negative sectional curvature such that the space $T^1X$ is not foliated by lifts of minimal surfaces, that is conjugated to the homogeneous foliation.

The first step for proving Theorem \ref{t.foliated_PP} is the following proposition.

\begin{prop}[First foliated Plateau problem]\label{th.first_PP}
Let $\xi_-,\xi_+$ be two distinct points of $\bord X$. Let $\cF^+$ be a foliation by round circles of $\bord X\setminus\{\xi_-,\xi_+\}$, oriented so that $\xi_-$ belongs to the exterior component of all these circles. Then the family $\{D_c\}_{c\in\cF^+}$ foliates $X$.
\end{prop}

\begin{proof}
Let us sketch the proof. For $c\in\cF^+$, consider the convex set $K_c$ that is bounded by $D_c$, so $\xi^-\in\bord K_c$ (see Theorem \ref{th_as_plateau_pb}). By the global geometric maximum principle (Theorem \ref{th.global_max}), these convex sets are nested. The strong geometric maximum principle (see Theorem \ref{th.strong_max_pcpl}) shows that the discs $D_c$ are disjoint.

We show that the family $\{D_c\}_{c\in\cF^+}$ covers the whole of $X$. Suppose there exists a closed ball $B$ disjoint from all discs $D_c$. Denote $I=\{c\in\cF^+;B\dans K_c\}$. This is a non-empty interval (oriented from $\xi^+$ to $\xi^-$), with non-empty complement. Indeed, by convexity $\bigcup_{c\in\cF^+}\overline{K_c}=\overline{X}\setminus\{\xi^+\}$, and by non-existence of pseudo-horospheric $k$-surfaces, $\bigcap_{c\in\cF^+} \overline{K_c}=\{\xi^-\}$.  The affirmation now follows from the fact that the family $K_c$ is nested and by continuity of $c\mapsto D_c$.

Note now that, for all $c$ in the closure of the complement of $I$, the interior of $B$ is contained in the complement of $K_c$. It follows that if $c$ is a boundary point of $I$, then $D_c$ meets $B$ non-trivially, which is absurd.
\end{proof}

\begin{proof}[Proof of Theorem \ref{t.foliated_PP}] Let us sketch the proof that if $c,c'$ are distinct oriented round circles of $\bord X$, then the Gauss lifts $\widehat{D_c}$ and $\widehat{D_{c'}}$ are disjoint. Recall Proposition \ref{prop.harmonic_pol}. If the Gauss lifts $\widehat{D_c}$ and $\widehat{D_{c'}}$ intersect at $(x,v)\in T^1X $ then, in a neighbourhood of $x$, the surface $D_{c'}$ looks like the graph of a harmonic polynomial of degree $d \geq 2$. Hence the intersection locus (as the zero set of this polynomial) is the union of $2d$ smooth segments that intersect at $x$. This implies that $\overline{D_c}\cap \overline{D_{c'}}$ is a graph embedded in $\overline{X}$ whose vertices have valency $\geq 4$ (if they belong to $X$) or $3$ (if they belong to $\bord X$). Euler's formula for planar graphs shows that such a graph must have a closed loop, which bounds two discs $\Omega'\dans \overline{D_{c'}}$ and $\Omega'\dans \overline{D_{c'}}$, whose union is a sphere $S$. Recall that by Theorem \ref{th.first_PP}, $D_c$ can be seen as the leaf of a foliation of $X$ by $k$-surfaces so the sphere $S$ must be interior tangent to some leaf of this foliation, which contradicts the geometric maximum principle.

The preceding argument, together with the invariance of the domain, proves that the set $\{\widehat{D_c};c,\,\text{oriented round circle}\}$ is open in $T^1X$. By the compactness theorem, this set is also closed. It therefore covers the whole of $T^1X$, yielding the desired foliation.
\end{proof}

\subsubsection{The homogeneous case}
Assume that $X=\Hyp$. In that case the foliation described in Theorem \ref{t.foliated_PP} is smoothly conjugated to a nice homogeneous model.

Let us first consider the frame bundle of $\Hyp$
$$F\Hyp=\{(x,v,e_1,e_2);x\in\Hyp,\,(v,e_1,e_2),\text{oriented orthonormal basis of } T_x\Hyp\}.$$

Note first that the natural action of $\Isom^+(\Hyp)\simeq\PSLc$ on the frame bundle $F\Hyp$ is free and transitive, so that, up to a choice of base-frame, $F\Hyp\simeq\PSLc$. Right-multiplication then defines a natural, free, right-action of $\PSL$ on $F\Hyp$, which we denote by $\alpha_0$. The orbit under this action of any point $\xi:=(v,e_1,e_2)$ is then simply the frame bundle $FP$ of the oriented totally geodesic plane $P$ passing through $p$, with outward-pointing normal $v$. Alternatively, we may identify $F\Hyp$ with the set of uniformizing maps of oriented totally geodesic planes of $\Hyp$.

This right action $\alpha_0$ is smoothly conjugated to the right action $\alpha_k$. Indeed let $FP$ be the frame bundle of an oriented totally geodesic plane $P$. Write $k=\tanh^2(R)$ and consider the time $R$ map of the \emph{frame flow} $\hat G_R:F_{h_0}\tilde X\to F_{h_0}\tilde X,(p,v,e_1,e_2)\mapsto(G_R(p,v),e_1',e_2')$, where $G_R$ denotes the time $R$ map of the geodesic flow and $(e_1',e_2')$ denote the parallel transport of $(e_1,e_2)$ along this flow. The image of the frame bundle $FP$ of the totally geodesic plane $P$ is simply the frame bundle $FD$ of the $k$-disk $D$ spanned by $\bord P$, and $\hat G_R$ is simply a uniformizing map for $D$. It follows that $\hat G_R$ conjugates $\alpha_0$ and $\alpha_k$. In summary, we have the following identification.

\begin{theorem}[The homogeneous model]\label{t.homogeneous_model}
The right action $\alpha_k$ of $\PSL$ on $F\Hyp$ is smoothly conjugated to the homogeneous action of $\PSL$ on $\PSLc$ by multiplication on the right.
\end{theorem}

Since $p:F\Hyp\rightarrow T^1\Hyp$ is a principal $\text{SO}(2)$-bundle, $T^1\Hyp$ likewise identifies with the quotient space $\PSLc/\text{SO}(2)$. The images of the orbits of the homogeneous $\PSL$-action under this identification are then the Gauss lifts of totally umbilical planes. It follows that the $k$-surface foliation of $T^1\Hyp$ is likewise smoothly conjugate to the homogeneous foliation $\PSLc/\text{SO}(2)$ by $\PSL$-orbits.

\subsection{Rigidity of foliated Plateau problems} 

We now review the rigidity properties of the object defined above.

\subsubsection{Rigidity of the $k$-surface foliation}\label{sss.rigidity_FPP}

We saw that the unit tangent bundle of a simply connected negatively curved surface can be identified with the set of oriented triples of its ideal boundary. We discuss a three dimensional analogue of this identification.

We let $\text{MD}$ denote the space of closed oriented marked round discs $(D,z)$ of $\C\PP^1$. This is a $5$-dimensional manifold that comes with a natural $2$-dimensional foliation $\ctG_2$ whose leaves are defined by the equivalence relation $(D,z)\sim(D',z')$ whenever $D'=D$.

Now consider the boundary correspondence $\alpha:\C\PP^1\to\bord X$ defined by a negatively curved metric $h$ on $M$. We can define the map $\Phi_{k,h}:\text{MD}\to T^1X$ as follows. Let $(D,z)\in\text{MD}$. Let $c=\alpha(\partial D)$ and $\xi^+=\alpha(z)$. The curve $c$ comes with a natural orientation, so $\xi^+$ belongs to the interior component of $c$. It spans a unique oriented $k$-disc $D_c\dans X$. Consider the unique element $(x,v)$ of its Gauss lift $\widehat{D_c}$ such that $\xi^+$ is the limit point of the geodesic ray directed by $v$. We define $\Phi_{k,h}(D,z)=(x,v)\in T^1X$. Theorem \ref{t.foliated_PP} may be rephrased as follows.

\begin{theorem}[\cite{ALS}, Theorem 2.1.5.]
The map $\Phi_{k,h}:\mathrm{MD}\to T^1X$ is a homeomorphism that conjugates the foliations $\ctG_2$ and $\ctF_{k,h}$, the $k$-surface foliation of $T^1X$.
\end{theorem}

This theorem immediately implies the rigidity of $k$-surface foliations.

\begin{theorem}
If $h_1$ and $h_2$ are two negatively curved metrics on $M$, then the map $\Phi_{k,h_2}\circ\Phi_{k,h_1}^{-1}$ is a homeomorphism of the unit tangent bundles that conjugates the respective $k$-surface foliations.
\end{theorem}

\subsubsection{Gromov's geodesic rigidity revisited} It is interesting to note that our foliation by $k$-surfaces may be used to give another proof of Gromov's geodesic rigidity theorem that provides a canonical conjugacy. Let us describe a $1$-dimensional foliation of the space MD introduced above. 

Given a marked disc $(D,z)$, we let $z'=R_{\partial D}(z)$ denote the image of $z$ by the inversion of the circle $\partial D$. Let $L_{(D,z)}$ be the one-parameter family of marked discs invariant by the group of hyperbolic Möbius maps fixing $z$ and $z'$. This is a curve in MD passing through $(D,z)$. The set of curves constructed this way forms a foliation $\ctG_1$ of MD parametrized by the set of pairs of points of $\C\PP^1$.

Consider once more the boundary correspondence $\alpha:\C\PP^1\to\bord X$. We can define a map $\Psi_{k,h}:\text{MD}\to T^1X$ as follows. Let $(D,z)\in\text{MD}$ and $z'=R_{\partial D}(z)$. Let $\xi=\alpha(z)$, $\xi'=\alpha(z')$ and $c=\alpha(\partial D)$ endowed with its natural orientation. Let $D_c$ be the $k$-disc spanned by $c$. By convexity the geodesic $(\xi',\xi)$ cuts $D_c$ at a unique point $x$. Let $v$ be the derivative of this geodesic (parametrized from $\xi^-$ to $\xi^+$) at $x$. We can define $\Psi_{k,h}(D,z)=(x,v)$ so in particular, by construction, $\Psi_{k,h}$ maps the curve $L_{(D,z)}$ onto the geodesic $(\xi^-,\xi^+)$.

\begin{theorem}[\cite{ALS}, Theorem 2.1.4.]
The map $\Psi_{k,h}:\mathrm{MD}\to T^1X$ is a homeomorphism that conjugates the foliation $\ctG_1$ and the geodesic foliation $\ctG$ of $T^1X$.
\end{theorem}

\section{Equidistribution, asymptotic counting and geometric rigidity}

In Section 2, we discussed the problem of counting closed geodesics in negative curvature. Here, we are interested in the analogous asymptotic counting problem for closed essential surfaces. We first address the problem of the existence of such surfaces.

\subsection{Essential surfaces, topological counting and quasi-Fuchsian groups}\label{ss.essential_QF}

\subsubsection{Quasi-Fuchsian groups} Let $C>1$. A $C$-quasi-Fuchsian group is a discrete and torsion-free subgroup $\Gamma$ of $\PSLc$ isomorphic to a cocompact lattice of $\PSL$ whose limit set $\bord \Gamma$ is a $C$-quasicircle. A $1$-quasi-Fuchsian group is nothing but a Fuchsian surface subgroup of $\PSLc$.

\subsubsection{Existence results}
The existence of surface subgroups in the fundamental group of a hyperbolic $3$-manifold was an open question until the work of Kahn-Markovi\'c (see \cite{KM1}). In this work, the authors show how to construct infinitely many closed surfaces in a closed hyperbolic $3$-manifold $M$ (see also Bergeron's Bourbaki seminar \cite{Bergeron} for a clear exposition). This theorem was generalized by H\"amenstadt \cite{HamenstadtKM} and by Kahn-Labourie-Mozes \cite{Kahn_Labourie_Mozes} (we refer for Kassel's Bourbaki seminar \cite{KasselBourbaki} for a clear exposition). The compact subgroups constructed by Kahn-Markovi\'c are quasi-Fuchsian, and can be made almost Fuchsian.

\begin{theorem}[Kahn-Markovi\'c's existence theorem \cite{KM1}]\label{th.KM}
Let $(M,h_0)$ be a closed hyperbolic $3$-manifold. Then for every $\eps>0$ there exists a hyperbolic surface $S$ and an immersion $\iota:S\to M$ such that
\begin{enumerate}
\item the induced map $\iota_\ast:\pi_1(S)\to\Pi$ is injective and the image is a $(1+\eps)$-quasi-fuchsian group;
\item $\iota$ is $(1+\eps)$-quasi-geodesic, i.e., the image of a complete geodesic of $S$ is a $(1+\eps)$-quasi-geodesic in $M$.
\end{enumerate}
\end{theorem}

\subsubsection{Topological counting} By a theorem of Thurston, given a positive integer $g$, there are finitely many conjugacy classes of surface subgroups of genus up to $g$ in the fundamental group of a closed hyperbolic $3$-manifold. We refer to \cite[\S 3.3]{LabourieBourbaki} for a proof of this fact using minimal surfaces. Kahn and Markovi\'c provide a precise topological asymptotic counting result showing that the number of such conjugacy classes grows superexponentially with genus.

Given a cocompact lattice $\Pi\leq\PSLc$, we let QF denote the set of conjugacy classes $[\Gamma]$ of quasi-fuchsian subgroups of $\Pi$. Given $[\Gamma]\in$ QF we let g$(\Gamma)$ denote the genus of any representative $\Gamma$.

\begin{theorem}[Kahn-Markovi\'c's topological counting \cite{KM2}]\label{th.KM2}
If $(M,h_0)$ is a closed hyperbolic $3$-manifold then
\begin{equation}\label{eq.topo_counting}\lim_{g\to\infty}\frac{1}{2g\log(g)}\log\#\{[\Gamma]\in\mathrm{QF};\,\mathrm{g}(\Gamma)\leq g\}=1.
\end{equation}
\end{theorem}

Kahn-Marković’s result shows that the above limit also holds if one counts all conjugacy classes of surface groups of genus $\leq g$ and not only the quasi-Fuchsian ones.

\subsection{Geometric asymptotic counting and rigidity results}
The purpose of \cite{CMN} and \cite{ALS} is to give a geometric asymptotic counting. The problem of finding geometric representatives to conjugacy classes of surface subgroups is treated there by solving asymptotic Plateau problems for minimal surfaces (in \cite{CMN}) or for $k$-surfaces (in \cite{ALS}).

Assume that $h$ is a Riemannian metric on $M$ such that $\sect_h\leq -1$, and fix a number  $0<k<1$.

\subsubsection{$k$-surface representatives}

Let $\Gamma$ be a quasi-Fuchsian subgroup of $\Pi$. Its limit set $\bord\Gamma$ is an oriented quasicircle so we can consider the $k$-disc $D_{k,h}(\bord \Gamma)$ that it spans. By uniqueness of the solution of the asymptotic Plateau problem, the disc $D_{k,h}(\bord \Gamma)$ is $\Gamma$-invariant. The quotient
$$S_{k,h}([\Gamma])=D_{k,h}(\bord \Gamma)/\Gamma\dans M,$$
is naturally a closed surface whose fundamental group represents $[\Gamma]$ by construction. This is the unique $k$-surface representative of $[\Gamma]$.

\subsubsection{The geometric counting}\label{sss.geo_counting}
We are now ready to state the main result of \cite{ALS}, which is a geometric asymptotic counting result. This result defines an entropy that counts closed quasi-Fuchsian $k$-surfaces according to their areas and that satisfies a rigid inequality, analogous to Besson-Courtois-Gallot's theorem \ref{t.BCG}.

\begin{defi}[Area entropy of $k$-surfaces]\label{def.entropy}
Let $h$ be a Riemannian metric with $\sect_h\leq -1$. Fix $0<k<1$ and define the area entropy of $k$-surfaces as
$$\mathrm{Ent}_{k}(M,h)=\liminf_{A\to\infty}\frac{1}{A\log(A)}\log\#\left\{[\Gamma]\in\mathrm{QF};\Area_h(S_{k,h}([\Gamma])\leq A\right\}.
$$
\end{defi}

\begin{theorem}[Geometric counting of $k$-surfaces]\label{th.main_result_geo_counting}
Let $h$ be a Riemannian metric with $\sect_h\leq -1$ and $0<k<1$. Then we have the following chain of inequalities
$$\frac{H(M,h)^2}{2\pi}\geq\mathrm{Ent}_k(M,h)\geq \mathrm{Ent}_k(M,h_0)=\frac{1-k}{2\pi},$$
where $H(M,h)$ denotes the topological entropy of the geodesic flow.

Moreover the equality $\mathrm{Ent}_k(M,h)=\mathrm{Ent}_k(M,h_0)$ holds if and only if $h$ and $h_0$ are isometric.
\end{theorem}

The upper bound was obtained by Calegari-Marques-Neves in \cite{CMN} and is a consequence of the two-dimensional analogue of Besson-Courtois-Gallot's inequality obtained by Katok in \cite{Katok_entropy_closed}. We refer to \cite[Section 7]{CMN} or \cite[Proposition 6.3]{LabourieBourbaki} for a complete argument.

The equality $\mathrm{Ent}_k(M,h_0)=\frac{1-k}{2\pi}$ easily follows from the Gauss-Bonnet theorem. Indeed by Gauss' equation \eqref{eq.Gauss_eq}, in hyperbolic geometry, $k$-surfaces have intrinsic curvature $k-1<0$. Hence the Gauss-Bonnet theorem implies that for a closed $k$-surface $S\dans(M,h_0)$ of genus $g$
$$g=\frac{1-k}{4\pi}\Area_{h_0}(S)+1.$$

Then one computes the entropy in this case by using the topological asymptotic counting of Kahn-Markovi\'c. This Gauss-Bonnet argument also proves the inequality $\mathrm{Ent}_k(M,h)\geq\mathrm{Ent}_k(M,h_0)$. Indeed, if $\sect_h\leq -1$ then the Gauss-Bonnet theorem and the Gauss equation together give for every conjugacy class of quasi-Fuchsian group $[\Gamma]$ 
\begin{eqnarray*}
(1-k)\Area_{h_0}(S_{k,h_0}([\Gamma]))&=&4\pi(\text{g}(\Gamma)-1)\\
 &=&\int_{S_{k,h}([\Gamma])}\left(\left|\sect_h|_{TS_{k,h}([\Gamma])}\right|-k\right)d\Area_h\\
 &\geq& (1-k)\Area_h(S_{k,h}([\Gamma])).
\end{eqnarray*}

This means that the area of a quasi-Fuchsian $k$-surface is smaller in variable curvature $\geq -1$ than in the hyperbolic case. Hence, there are more closed quasi-Fuchsian $k$-surfaces with small area in $(M,h)$ than in $(M,h_0)$. The inequality follows. The rigidity of the equality is the interesting result. Note that we get
\begin{equation}\label{eq.equality_case_curv_-1}
\Area_{h_0}(S_{k,h_0}([\Gamma]))=\Area_{h}(S_{k,h}([\Gamma]))\Longrightarrow \sect_h|_{TS_{k,h}}=-1.
\end{equation}

\subsubsection{Rigidity of the hyperbolic marked area spectrum}\label{sss.statement_rigidity_MAS}

 We state another rigidity result which is more in the spirit of the marked length spectrum rigidity. We define the marked area spectrum of $k$-surfaces for the metric $h$ as

\begin{equation}\label{eq.area_spectrum}\MAS_{k,h}:\text{QF}\to\R^+,\,\,\,\,\,\,\,\,\,[\Gamma]\mapsto\Area_h(S_{k,h}([\Gamma])).
\end{equation}
The following rigidity theorem was proven in \cite{ALS}. It is an analogue of Hamenst\"adt's result (Theorem \ref{th.ham_MLS}).

\begin{theorem}[Rigidity of the hyperbolic marked area spectrum]\label{th.rigidity_MAS}
Let $(M,h_0)$ be a closed hyperbolic $3$-manifold. Let $h$ be a Riemannian metric on $M$ with $\sect_h\leq -1$ and $k\in(0,1)$. Then $\MAS_{k,h}=\MAS_{k,h_0}$ if and only if $h$ and $h_0$ are isometric.
\end{theorem}

Note that by \eqref{eq.equality_case_curv_-1} if $\MAS_{k,h}=\MAS_{k,h_0}$, the sectional curvature of $h$ must be equal to $-1$ in restriction to all tangent planes of all closed quasi-Fuchsian surfaces in $M$. The key point is to prove the equidistribution of tangent planes of quasi-Fuchsian $k$-surfaces to prove that the sectional curvature equals $-1$ in \emph{all} planes of $TM$. This would imply that the metric $h$ is hyperbolic. Mostow's rigidity theorem then proves that $h$ must be isometric to $h_0$.

\subsection{Equidistribution of quasi-Fuchsian $k$-surfaces}
We use the formalism of foliated Plateau problems introduced in the previous section and discuss the equidistribution property needed to prove Theorems \ref{th.main_result_geo_counting} and \ref{th.rigidity_MAS}.

\subsubsection{Conformal currents, $\PSL$-invariant measures and laminar measures}\label{sss.conf_currents_PSL_inv_laminar_measures} The main reason we introduced the fibered Plateau problem in \S \ref{sss.fibered} is that it enables us to generalize the identification between geodesic currents and invariant measures by the geodesic flow (see Bonahon \cite{Bonahon_laminations} for example) in our context. Recall that geodesic currents are $\Pi$-invariant measures on the set of geodesics, i.e., the set $\bord^{(2)}X$ of pairs of points of the ideal boundary, and that they correspond to measures on $T^1X$ both invariant by $\Pi$ and by the geodesic flow. See \S \ref{sss.BM_measures} for the example of the Bowen-Margulis measure.

Consider the bundles $\bord:\text{FKD}_{k,h}(C)\to\text{QC}^+(C)$ and $\bord:\text{MKD}_{k,h}(C)\to\text{QC}^+(C)$ defined in \S \ref{sss.fibered}. Recall that the space $\QC^+(C)$ is separable and locally compact. We say that a Borel measure $\mu$ on $\FKD_{k,h}(C)$ (resp. $\MKD_{k,h}(C)$) is $\PSL$-invariant whenever there exists a Borel measure $m$ on $\QC^+(C)$ such that, in every trivializing chart $U\times F$ of $\bord$,
\begin{equation*}
\mu|_{U\times F}=m|_U\times\lambda_F\ ,
\end{equation*}
where $\lambda_F$ denotes the Haar measure on $\PSL$ (resp. the hyperbolic area of $\D$). Borel regular $\PSL$-invariant measures on $\FKD(C)$ (resp. $\MKD(C)$) are trivially in one-to-one correspondence with Borel regular measures on $\QC^+$ (see \cite[Lemma 4.2.1.]{ALS}). We call $m$ the \emph{factor} of $\mu$, and we call $\mu$ the \emph{lift} of $m$. By construction, a $\PSL$-invariant measure $\mu$ is also $\Pi$-invariant if and only if its factor $m$ is.

In other words, we have natural bijective correspondences between each the following classes of objects (see \cite[\S 4.2]{ALS} for a slightly different formalism).
\begin{enumerate}
\item $\Pi$-invariant Borel regular measures $\hat m$ on the space $\text{QC}^+(C)$ of oriented $C$-quasicircles; \emph{these are the analogues in our context of geodesic currents}.
\item $(\Pi,\PSL)$ bi-invariant Borel measures $\hat\nu$ over $\FKD_{k,h}(C)$; \emph{these are the analogues in our context of measures invariant under the geodesic flow}.
\item $\Pi$-invariant Borel measures $\hat\mu$ on $\text{MKD}_{k,h}(C)$ which are totally invariant for the lamination $\cL_{k,h}$, that is, whose disintegration on each leaf of $\cL_{k,h}$ is a multiple of the hyperbolic area; \emph{these are the analogues in our context of totally invariant measures of the geodesic foliation}.
\end{enumerate}

According to Labourie's terminology (see \cite{LabourieBourbaki,MN_currents}), measures of the first class are called \emph{conformal currents}, measures of the second are called \emph{$(\Pi,\PSL)$-bi-invariant measures}, and measures of the third are called \emph{laminar measures}.

Conformal currents were the main topic of the recent and very interesting work \cite{MN_currents}, where the authors show in particular how to combine them with geodesic currents. We must also cite \cite{KMS} for an application of conformal currents in the study of random minimal surfaces. We expect that similar results should hold for $k$-surfaces.

An important example is given by the generalization of Dirac masses supported by closed geodesics. Let $[\Gamma]\in\text{QF}$ and $S=S_{k,h}([\Gamma])$. Let $C$ the quasi-conformality constant of $c:=\bord\Gamma$. Let $\delta(c)$ be the Dirac mass on $\text{QC}^+(C)$ supported by $c$. We now introduce the following conformal measure, bi-invariant measure, and laminated measure.\\

\noindent\emph{The conformal current associated to $S$}. We define
\begin{equation*}
\hat m(c)=\frac{1}{2\pi|\chi(S)|}\sum_{\gamma\in\Pi/\Gamma}\delta(\gamma\cdot c)\ ,
\end{equation*}
where $\gamma\cdot c$ abusively denotes the image of $c$ under any representative of $\gamma\in\Pi/\Gamma$. By construction, $m(c)$ is infinite, atomic and $\Pi$-invariant. We call it the \emph{conformal current} associated to $c$.\\

\noindent\emph{The bi-invariant measure associated to $S$}. Recall that the frame bundle $FD_c$ of $D_c$ naturally identifies, up to a choice of base point, with $\PSL$. Let $\omega_{k,h}(c)$ denote the image of the Haar measure under this identification. We define the $(\Pi,\PSL)$-bi-invariant measure associated to $S$ over $\FKD_{k,h}(C)$ by
\begin{equation*}
\hat\nu_{k,h}(c)=\frac{1}{2\pi|\chi(S)|}\sum_{\gamma\in\Pi/\Gamma}\omega_{k,h}(\gamma.c)\ .
\end{equation*}
This measure is analogous to the lift to $S_h \tilde X$ of the Dirac mass supported by a periodic orbit of the geodesic flow.\\

\noindent \emph{The laminar measure associated to $S$.} Let $\lambda_{k,h}(c)$ denote the area measure of the Poincar\'e metric of the disk $D_c$. The laminar measure associated to $S$ is the measure
\begin{equation*}
\hat\mu_{k,h}(c)=\frac{1}{2\pi|\chi(S)|}\sum_{\gamma\in\Pi/\Gamma}\lambda_{k,h}(\gamma.c)\ .
\end{equation*}
The quotient of this measure under the action of $\Pi$ yields a Borel regular measure $\mu_{k,h}$ over $\MKD_{k,h}(C)/\Pi$ supported on $\hat S$, where here $\hat S$ is viewed as a closed leaf of $\cL_{k,h}$.

Note that the measures on $\text{QC}^+(C)$ do not depend of the metric $h$, nor on the constant $k$. Only the $\PSL$-invariant measure $\hat\mu_{k,h}(c)$ does. This allows one to change the metric and work with the hyperbolic metric when needed. This is useful when one needs to apply Ratner's theory.

\subsubsection{Ratner's theorem and a criterion for equidistribution}

We will follow the main idea of \cite{CMN} (see also \cite{LabourieBourbaki}) and consider limits of measures $\hat\mu_{k,h}(c)$ when $c$ tends to a round circle $c_0$.

\begin{defi}[Ergodic conformal currents]
A conformal current $\hat m$ in $\QC^+(C)$ is said to be \emph{ergodic} if every $\Pi$-invariant measurable function on $\QC^+(C)$ is constant $\hat m$-almost everywhere.
\end{defi}

Note that the space $\cC^+$ of round circles in $\C\PP^1$ is a $3$-dimensional manifold carrying a natural Haar measure which we denote by $\Leb$.

\begin{theorem}[Classification of ergodic conformal currents]
\label{thm.Ratner} For every ergodic conformal current $\hat m$ on $\cC^+$, the following dichotomy holds.
\begin{itemize}
\item Either $\hat m$ is proportional to $\Leb$.
\item Or $\hat m$ is the conformal current of a closed Fuchsian surface.
\end{itemize}
\end{theorem}

\begin{proof}
We work with the hyperbolic metric $h:=h_0$ on $\tilde{X}$. Let $\hat\nu_0$ be an ergodic measure over $\PSLc\simeq F_{h_0}\tilde{X}$ which is both invariant under the left action of $\Pi$, and under the right action by multiplication of $\PSL$. By Ratner's classification theorem \cite{Ratner_Duke} (see \cite{Einsiedler} for a simple proof for the case of $\PSL$-actions), this measure is \emph{homogeneous}. That is, it is the Haar measure supported on some closed orbit $v_0G$ of some closed, connected Lie subgroup $G\leq\PSLc$ containing $\PSL$. Since the only closed and connected Lie subgroups of $\PSLc$ containing $\PSL$ are $\PSL$ and $\PSLc$ itself, this yields the following dichotomy:
\begin{itemize}
\item either $\nu_0$ is the Haar measure of $F_{h_0}\tilde X/\Pi$;
\item or $\nu_0$ is the Haar measure of the frame bundle of some closed, totally-umbilic surface of $\tilde X/\Pi$.
\end{itemize}
The result now follows.
\end{proof}

We deduce the following dichotomy for laminar measures.

\begin{theorem}[Dichotomy for laminar measures]\label{thm.dichotomy_laminar} There exists a unique ergodic laminar probability measure $\mu$ on $T^1M$ with full support. All other ergodic laminar measures on $T^1M$ are associated to closed Fuchsian $k$-surfaces.
\end{theorem}

\begin{remark}
When $\Pi$ contains no Fuchsian subgroup, there exists a unique laminar probability measure on $T^1M$ which is fully supported.
\end{remark}

\begin{proof}
Let $\hat\mu$ denote the lift of $\mu$ to $T^1M$ and let $\hat m$ denote its associated conformal current on $\cC^+$. If $\mu$ has full support, then so too does $\hat m$ so that, by Theorem \ref{thm.Ratner}, $\hat m$ is proportional to $\Leb$, and the first assertion follows. Likewise, if $\mu$ does not have full support, then, by Theorem \ref{thm.Ratner} again, $\hat\mu$ is the laminar measure associated to some closed Fuchsian $k$-surface, as in \S \ref{sss.conf_currents_PSL_inv_laminar_measures}, and the second assertion follows.
\end{proof}

\subsubsection{Equidistribution and Kahn-Markovi\'c's sequences} The following facts are the key to our rigidity results.

\begin{lemma}
Let $\Gamma_n$ be a sequence of $C_n$-quasi-Fuchsian subgroups of $\Pi$ and $S_n=S_{k,h}([\Gamma_n])$ be their $k$-surfaces representatives in $(M,h)$. Assume that $\lim_{n\to\infty}C_n=1$. Then any accumulation point of $\hat\mu_{k,h}(c_n)$ is a laminar measure supported on $\MKD_{k,h}(1)\simeq T^1X$.
\end{lemma}

Kahn-Markovi\'c's proof of Theorem \ref{th.KM} provides a sequence of surfaces with remarkable properties. In particular, the following theorem is relevant to us.

\begin{theorem}\label{t.equidistrib}
There exists a sequence $\Gamma_n$ of $C_n$-quasi-Fuchsian subgroups of $\Pi$ satisfying the following properties
\begin{enumerate}
\item $\lim_{n\to\infty}C_n=1$;
\item the sequence of invariant measures $\hat \mu_{k,h}(\bord\Gamma_n)$ converges to a laminar measure $\hat \mu_\infty$ on $T^1X$ with total support.
\end{enumerate}
\end{theorem}

The proof of this theorem may be found in \cite{ALS} and consists in showing that the $k$-surfaces representing Kahn-Markovi\'c's surfaces obtained on \cite{KM1} do not accumulate in closed surfaces. The result follows from the dichotomy for $\PSL$-invariant measures in $T^1M$ (Theorem \ref{thm.Ratner}).

\subsubsection{Rigidity of the hyperbolic marked area spectrum of $k$-surfaces} We will only prove Theorem \ref{th.rigidity_MAS}. Assume that we have an equality between the marked area spectra $\MAS_{k,h}=\MAS_{k,h_0}$. As we remarked in \S \ref{sss.statement_rigidity_MAS}, this implies that for every quasi-Fuchsian $k$-surface $S$ the sectional curvature equals $-1$ in $TS$. 

Define the function $\sigma$ associating to $(D,p)\in\text{MKD}_{k,h}(C)$ (for $C\geq 1$) the sectional curvature of $D$ in restriction to $T_pD$. When $C=1$ this function may be thought of as the function  $\sect_h:T^1X\to\R$ associating to $v\in T^1M$ the sectional curvature in $v^\perp$, i.e., in the tangent space of the leaf containing $v$.

Consider the Kahn-Markovi\'c's sequence given in Theorem \ref{t.equidistrib}. Then the function $\sigma=-1$, $\hat\mu_{k,h}(\bord\Gamma_n)$-everywhere. Hence, we have $\sigma =-1$, $\hat\mu_\infty$-almost everywhere. Since $\hat \mu_\infty$ has full support and $\sect_h$ is continuous, we obtain $\sect_h=-1$ in $T^1X$. This implies that the sectional curvature of $M$ equals $-1$ everywhere. By Mostow's rigidity theorem $h$ and $h_0$ are isometric.

In a recent work \cite{ALS2} we obtained a number of results that go beyond the rigidity of the marked area spectrum, looking at the \emph{marked energy spectrum} i.e. the map 
\begin{equation}\label{eq.energy_spectrum}\text{MES}_{k,h}:\text{QF}\to\R^+,\,\,\,\,\,\,\,\,\,[\Gamma]\mapsto\int_{S_{k,h}([\Gamma])}Hd\Area_h,
\end{equation}
where $H$ is the mean curvature of the surface. The integral of the mean curvature over the surface, called the \emph{energy} of the surface, is nothing but the area of the Gauss lift of $S$ in $T^1M$ calculated with the Sasaki metric. We obtain two result
\begin{enumerate}
\item the rigidity of the hyperbolic marked energy spectrum for negatively curved metric subject to a \emph{lower bound} $-1\leq\sect_h$;
\item the fact that the marked area and energy spectra are asymptotic (see \cite{ALS2} for more details) if and only if the sectional curvature is constant.
\end{enumerate}
These results can be compared with the classical measure rigidity theorems by Katok and Ledrappier \cite{Katok_entropy_closed,Katok4,Ledrappier_har_BM}.
 It would be tempting to think that they fit into a thermodynamical theory of $k$-surfaces, analogous to that of the geodesic flow.

\section{Some questions}
The analogy between the dynamics of $k$-surfaces and of the geodesic flow is a source of numerous questions that seem interesting.  We finish this paper by asking some of them.

\subsection{Rigidity of the $k$-surface foliation}
The foliation $\cF_{k,h}$ constructed in \S \ref{sss.FPP} is only continuous with smooth leaves. In the spirit of \cite{BFL} it is natural to ask
\begin{question}
Assume that the foliation $\cF_{k,h}$ is smooth. Is it true that $X$ has constant sectional curvature?
\end{question}

In another direction it is natural to compare the equivalence between $k$-surfaces foliations, with orbit equivalence for the geodesic flow given by Gromov's geodesic rigidity theorem.

\begin{question}
Assume that the homeomorphisms $\Phi_{k,h}$ and $\Psi_{k,h}$ described in \S \ref{sss.rigidity_FPP} coincide. Is it true that $X$ has constant sectional curvature?
\end{question}

\subsection{Boundary rigidity problem} The boundary rigidity problem asks whether the Riemannian metric on a \emph{simple manifold} can be determined by the restriction of the Riemannian distance between points at the boundary. Recall that a simple manifold is a Riemannian ball with strictly convex boundary without conjugate points such that any pair of points in the boundary can be connected by a unique minimizing geodesic. This applies in particular to any negatively curved metric with strictly convex boundary. See for example \cite{CGL,Wen}.

The following question is a $2$-dimensional analogue of this boundary question. We also refer to \cite{ABN} for a related question for area minimizing surfaces.

We consider a ball $B$ in hyperbolic $3$-space $\Hyp$. We consider the set $\cC^+(\partial B)$ of oriented round circles of $\partial B$. Fax a negatively curved metric $h$ on $B$ with $\sect_h\leq -1$ and a number $k\in(0,1)$. For any $c\in \cC^+(\partial B)$ there exists a unique $k$-disc $D_c\in B$ such that $\partial D_c=c$. we can consider the marked boundary data
$$\MAS_{h,k,\partial B}:\cC^+(\partial B)\to\R^+,\,\,\,\,\,\,\,\,\,c\mapsto\Area_h(D_c).$$

\begin{question}
Assume that $h$ and $h_0$ have the same marked boundary datas. Are $h$ and $h_0$ isometric?
\end{question}

Even a local rigidity in the spirit of \cite{CGL} would be interesting. The difficulty is the lack of tools from hyperbolic dynamics.

\subsection{Equidistribution of closed surfaces, thermodynamical formalism and cohomological equations} The theory of geodesic currents in negative curvature is well understood. Basically they are in bijective correspondence with cohomology classes of Hölder continuous potentials in $T^1M$. We refer to Ledrappier's paper \cite{Ledrappier_bord} for an exposition of this theory.

By contrast, the measures that arise as limits of quasi-Fuchsian surfaces are not well understood. A natural question is if there is an analogue of Bowen-Margulis' measure in our context. We refer to \cite{KMS} for a systematic study for the case of minimal surfaces. Let $\cS_{k,h,\eps}$ be the family of $(1+\eps)$-quasi-Fuchsian $k$-surfaces. We let $\delta(S)$ denote the measure in $T^1M$ supported by the Gauss-lift of $S$ induced by its hyperbolic area.

\begin{question}
Consider the following probability measures
$$\frac{1}{\#\{S\in\cS_{k,h,\eps};\Area_h(S)\leq R\}}\sum_{S\in\cS_{k,h,\eps};\Area_h(S)\leq R} \delta(S).$$
What can be said about the accumulation points of these measures as $R\to\infty$ and $\eps\to 0$?
\end{question}

As we mentionned before, it would be interesting to read the previous questions through the eye of thermodynamical formalism theory. A concrete starting point for establishing a theormodynamical formalism for $k$-surfaces, would be to get an analogous of Liv\v{s}ic's and Guillemin-Kazhdan's theorems.

\begin{question}
What can be said about a Hölder continuous function $F:T^1M\to\R$  that integrates $0$ on the Gauss lifts of all quasi-Fuchsian $k$-surfaces?

What can be said about a closed $2$-form that integrates $0$ on all quasi-Fuchsian $k$-surfaces?
\end{question} 

A subtle and nontrivial point in the question above is that of the regularity of objects that we consider. We might want to apply cohomological result (such as Liv\v{s}ic's) to Hölder continuous objects like in the proof of Theorem \ref{th.cohomology}.

\subsection{Beyond closed quasi-Fuchsian $k$-surfaces}

Actually, in \cite{LabourieAnnals}, Labourie uses  ideas from thermodynamical formalism and coding builds many closed $k$-surfaces. These $k$-surfaces are not quasi-Fuchsian, might not be incompressible and equidistribute in the huge phase space described in \S \ref{sss.Lab_lam}. In particular there are closed $k$-surfaces that limit to closed geodesics.

\begin{question}
Are incompressible $k$-surfaces in $M$ abundant in the space of all closed $k$-surfaces? Is it possible to define a more general asymptotic counting functional that takes into account these non incompressible $k$-surfaces? Would that make sense to talk about rigidity?
\end{question}

\begin{flushleft}
{\scshape S\'ebastien Alvarez}\\
	CMAT, Facultad de Ciencias, Universidad de la Rep\'ublica\\
	Igua 4225 esq. Mataojo. Montevideo, Uruguay.\\
	email: \texttt{salvarez@cmat.edu.uy}

\end{flushleft}

\end{document}